\documentclass[11pt]{article}

\usepackage[top=1in,bottom=1in,left=1.25in,right=1.25in]{geometry}

\usepackage[absolute,overlay]{textpos}

\usepackage{graphicx}
\usepackage{ragged2e}
\usepackage{tabto}

\usepackage{amsmath}
    \usepackage{amsthm}

\usepackage{float}
\usepackage{multirow}
\usepackage{tikz}
\usepackage{graphics}
\usepackage{color}
\usepackage{url}
\usepackage{afterpage}
\usepackage{multicol}
\usepackage[font=footnotesize]{caption}
\usepackage[font=footnotesize]{subcaption}
\usepackage{marginnote}
\usepackage{tcolorbox}

\usepackage{newtxtext}
\usepackage{newtxmath}

\usepackage[sf,bf,medium]{titlesec}
\usepackage{pgfplots}

\usepackage{bm}
\usepackage{booktabs}
\usepackage{siunitx}
\usepackage{isomath}

\usepackage[plainpages=false, colorlinks=true,
   citecolor=blue, filecolor=blue, linkcolor=blue,
   urlcolor=blue]{hyperref}
\usepackage{enumitem}

\usepackage{multicol}
\newcommand\vct[1]{\bm #1}

\newcommand{\lp}{\left(}
\newcommand{\rp}{\right)}

\usepackage[sort,compress]{cite}

\newcommand\bbR{\mathbb R}

\newcommand\bx{\boldsymbol{x}}

\newcommand\utrue{u_\text{true}}
\newcommand\chebfun{\texttt{Chebfun}}

\newtheorem{theorem}{\sffamily Theorem}
\newtheorem{remark}{\sffamily Remark}
\newtheorem{definition}{\sffamily Definition}
\newtheorem{problem}{\sffamily Problem}
\newtheorem{example}{\sffamily Example}
\newtheorem{lemma}{\sffamily Lemma}
\newtheorem{corollary}{\sffamily Corollary}[theorem]
\newtheorem{conjecture}{\sffamily Conjecture}

\newcommand{\cO}{\mathcal O}

\newcommand{\surflap}{\Delta_\Gamma}
\newcommand{\surfdiv}{\nabla_\Gamma \cdot}
\newcommand{\surfgrad}{\nabla_\Gamma}
\newcommand{\gradg}{\surfgrad}
\newcommand{\divg}{\surfdiv}
\newcommand{\LB}{\surflap}
\newcommand{\LBi}{\Delta_{\Gamma_i}}

\numberwithin{equation}{section}

\newcommand{\bs}{\boldsymbol}

\newcommand{\lsub}{L}
\newcommand{\ysub}{Y}
\newcommand{\GL}{G_{\lsub}}
\newcommand{\SL}{\mathcal S_{\lsub}}

\newcommand{\GY}{G_{\ysub}}

\newcommand{\Hs}{H^{s}(\Gamma)}

\newcommand{\Ht}{\bs L^{2}_t(\Gamma)}

\newcommand{\Ho}{H^{1}(\Gamma)}
\newcommand{\Hm}{H^{-1}(\Gamma)}
\newcommand{\Hoo}{H^{1}_{mz}(\Gamma)}
\newcommand{\Hmo}{H^{-1}_{mz}(\Gamma)}
\newcommand{\Hso}{H^{s}_{mz}(\Gamma)}
\newcommand{\Htg}{\mathcal{H}^2(\Gamma)}

\newcommand{\Iper}{[0,L]}
\newcommand{\Lr}{L^r(\Iper)}

\newcommand{\binormal}{\bs b_i}
\newcommand{\binormalj}{\bs b_j}
\newcommand{\half}{\frac{1}{2}}

\newcommand{\shat}{\bs{\hat s}}
\newcommand{\thetahat}{\bs{\hat \theta}}
\newcommand{\tr}{\operatorname{tr}}

\renewcommand{\phi}{\varphi}
\newcommand{\interior}{\mathrm{o}}

\binoppenalty=\maxdimen
\relpenalty=\maxdimen

\usetikzlibrary{arrows}
\usetikzlibrary{calc}
\tikzset{
    partial ellipse/.style args={#1:#2:#3}{
        insert path={+ (#1:#3) arc (#1:#2:#3)}
    }
}

\begin{document}

\begin{titlepage}

  \raggedleft
  {\sffamily \bfseries STATUS: arXiv pre-print}
  
  \hrulefill
  
  \raggedright


  

 



  \begin{textblock*}{\linewidth}(1.25in,2in) 
    {\LARGE \sffamily \bfseries An interface formulation of
      the Laplace-Beltrami\\
      \vspace{.25\baselineskip}
      problem on piecewise smooth surfaces}
  \end{textblock*}

  \vspace{1.5in}
  Tristan Goodwill\footnote{Research supported in part by the Research Training Group in Modeling and Simulation funded by the National Science Foundation via grant
    RTG/DMS-1646339 and by the Office of Naval Research under
    awards \#N00014-18-1-2307 and \#N00014-21-1-2383.}\\ \textit{\small Courant Institute, NYU\\ New
    York, NY 10012}\\ \texttt{\small tg1644@nyu.edu}

  \vspace{\baselineskip}
  Michael O'Neil\footnote{Research supported in part by
  the Simons Foundation/SFARI (560651, AB) and by the Office of Naval Research under awards \#N00014-18-1-2307 and \#N00014-21-1-2383.}\\
  \textit{\small Courant Institute, NYU\\
    New York, NY 10012}\\
  \texttt{\small oneil@cims.nyu.edu}

  \begin{textblock*}{\linewidth}(1.25in,7in) 
    \today
  \end{textblock*}

  \begin{textblock*}{\linewidth}(1.25in,8.25in) 
    \includegraphics[width=4in]{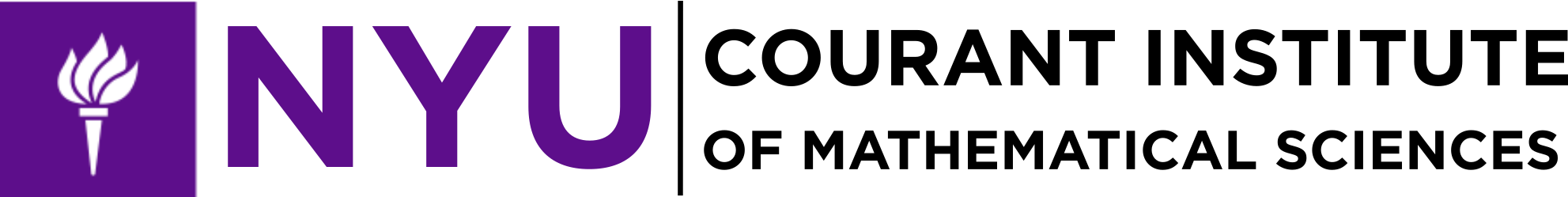}
  \end{textblock*}

  
\end{titlepage}

\begin{abstract}
  The Laplace-Beltrami problem on closed surfaces embedded in three
  dimensions arises in many areas of physics, including molecular
  dynamics (surface diffusion), electromagnetics (harmonic vector
  fields), and fluid dynamics (vesicle deformation). In particular,
  the Hodge decomposition of vector fields tangent to a surface can be
  computed by solving a sequence of Laplace-Beltrami problems. Such
  decompositions are very important in magnetostatic calculations and
  in various plasma and fluid flow problems. In this work we develop
  $L^2$-invertibility theory for the Laplace-Beltrami
  operator on piecewise smooth surfaces, extending earlier weak
  formulations and integral equation approaches on smooth surfaces.
  Furthermore, we reformulate the weak form of the problem as
  an interface problem with continuity conditions across edges of
  adjacent piecewise smooth panels of the surface. 
  We then provide high-order numerical examples along surfaces of
  revolution to support our analysis, and discuss numerical extensions
  to general surfaces embedded in three dimensions.  \\
  

\end{abstract}

\tableofcontents

\newpage

\section{Introduction}\label{sec:intro}

The Laplace-Beltrami operator is a second-order elliptic differential
operator defined along Riemannian manifolds in arbitrary
dimensions. Practically, it can be thought of as the extension of the
Laplace operator to curved
surfaces~\cite{frankel2011,nedelec2001}. For the moment, let~$\Gamma$
be a smooth closed surface embedded in three dimensions, and
let~$\surfdiv$ and~$\surfgrad$ be the intrinsic surface divergence and
surface gradient operators along~$\Gamma$ (these are defined more
carefully later on in the manuscript). Then the Laplace-Beltrami
operator, also referred to as the \emph{surface Laplacian}, is given
as
\begin{equation}
  \surflap = \surfdiv \surfgrad.
\end{equation}
The Laplace-Beltrami operator is particularly useful for problems in
electromagnetics since it allows for the explicit construction of
tangential vector fields along multiply-connected surfaces in terms of
their gradient, divergence-free, and harmonic components. For example, any
smooth tangential vector field~$\vct{F}$ along~$\Gamma$ can be written as
\begin{equation}
  \label{eq:hdg}
  \vct{F} = \surfgrad \alpha  + \vct{n} \times
  \surfgrad \beta  + \vct{H},
\end{equation}
where~$\vct{n}$ is the unit normal to~$\Gamma$, the functions $\alpha,\beta$ are
smooth and scalar-valued along~$\Gamma$, and~$\vct{H}$ is a harmonic vector
field tangent to~$\Gamma$, i.e. $\surfdiv \vct{H} = 0$ and
$\surfdiv (\vct{n} \times \vct{H}) = 0$.
See~\cite{Epstein2009,Epstein2013,ONeil2018,nedelec2001} for more details
regarding such decompositions and a proof of their uniqueness. Conversely, if
the vector field~$\vct{F}$ above is known (but not its individual components),
then, for example, its solenoidal (i.e. divergence-free) component
involving~$\beta$ can be computed by taking~$\surfdiv \vct{n} \times$ of each
side of~\eqref{eq:hdg}, yielding
\begin{equation}
\surflap \beta = -\surfdiv \vct{n} \times \vct{F}.
\end{equation}
Solving the above PDE along~$\Gamma$ requires inverting the
Laplace-Beltrami operator. Applications of this problem are myriad in
electromagnetics, as mentioned, as well as plasma
physics~\cite{Malhotra2019}, vesicle deformation in biological
flows~\cite{rahimian2010petascale,veerapaneni2011fast}, surface
diffusion~\cite{bansch2005finite,escher1998surface}, and computational
geometry~\cite{kromer2018} and computer vision~\cite{angenent1999brain}.

Along smooth general surfaces, there have been several numerical
methods proposed to solve the Laplace-Beltrami problem: finite element
methods~\cite{Dziuk1988,bansch2005finite,Burman2017,Bonito2020,demlow2007adaptive}
(including eigenfunction computations~\cite{bonito2018}), the so-called
\emph{virtual element
method}~\cite{frittelli2018virtual,beirao2014hitchhiker}, differencing
methods~\cite{wang2018modified}, and integral equation
methods~\cite{ONeil2018,kropinski2014fast}. Of course, in specialized
geometries, such as axisymmetric ones, separation of variables can be
used to simplify the problem. In~\cite{Epstein2019,o2018integral},
after separation of variables, an integral equation approach was used
along the one-dimensional generating curve of the axisymmetric surface
resulting in high-order convergence for the inversion of the surface
Laplacian operator. Boundary-value problems for the Laplace-Beltrami
problem were addressed
in~\cite{kropinski2014fast,kropinski2016integral} via a parametrix
method and associated discretization via projection of the sphere.

However, while the related problem of electromagnetic scattering from
non-smooth surfaces has been studied in some
detail~\cite{Costabel,Buffa2002}, the Laplace-Beltrami problem on
piecewise-smooth surfaces (or more generally Lipschitz surfaces) has
not received as much attention. In order to extend various numerical
methods in computational electromagnetics, namely integral equation
methods based on generalized Debye source
representations~\cite{Epstein2013,Epstein2019,Chernokozhin2013,Epstein2009},
a thorough understanding of the Laplace-Beltrami problem on non-smooth
surfaces is required (as well as high-order accurate methods for
solving the problem numerically).  This work aims to offer a
mathematical discussion of the Laplace-Beltrami problem along
piecewise smooth surfaces in the~$L^2$ setting, which is compatible
with many modern numerical methods (i.e. PDE vs. integral equation
methods, Galerkin vs. Nystr\"om discretizations). In particular, our
main purpose for focusing on the $L^2$ theory of this problem is to
develop fast robust numerical solvers that can be incorporated into
computational electromagnetics codes. These solvers often rely on
iterative methods whose behavior is coupled with the spectrum of the
operator, which depends strongly on the type of discretization used. A
suitable~$L^2$ embedding of the problem~\cite{Bremer2012} is
frequently the most straightforward way to obtain an accurate
approximation to the spectrum of the continuous problem.

Previously, the regularity of solutions to the Laplace-Beltrami
problem in the general Lipschitz setting was discussed
in~\cite{Gesztesy2011}, and the regularity of solutions on polyhedral surfaces was discussed in \cite{Buffa2002}. Some special-case numerical methods on
polyhedral surfaces were
presented in~\cite{hildebrandt2011,wardetzky2007}. These methods are
almost exclusively developed in the finite element setting whereby the
use of polyhedral surfaces and linear finite elements reduce the
continuous problem to a discrete system that can be studied in detail.
Similar techniques have been used in the electromagnetics community
for analyzing Hodge decompositions on
polyhedra~\cite{buffa2003,Buffa2001}.

The main result of our manuscript is Theorem~\ref{thm:strongLB},
which proves a reformulation of the weak form of the Laplace-Beltrami
problem~$\surflap u = f$ as an \emph{interface} or \emph{transmission}
problem on the curved but piecewise smooth surface~$\Gamma$ with interface, or continuity, conditions along
the edges of the smooth panels. This reformulation is akin
to what are sometimes referred to as dielectric transmission problems
or Laplace interface problems in the plane. In fact, our reformulation
takes advantage of previous results for exactly this problem.

The outline of our approach and the paper is as follows: in
Section~\ref{sec:analytical} we recall some standard results
regarding the weak Laplace-Beltrami problem on smooth and Lipschitz
surfaces. In Section~\ref{sec:strong}, we extend those results to an interface formulation on piecewise smooth
surfaces. In Section~\ref{sec:cone}, we extend these results to a flat cone and discuss how they might be extended to other surfaces with conic singularities. As an example of a case where a simple numerical scheme exists, in Section~\ref{sec:ode} we reformulate interface form of the
Laplace-Beltrami problem as a sequence of decoupled ODEs along a
surface of revolution, and introduce an integral equation
reformulation of said ODEs. An associated high-order accurate
numerical solver is then discussed. Section~\ref{sec:numerical}
contains various numerical experiments validating our formulation and
demonstrating the accuracy of our solver. Finally, in
Section~\ref{sec:conclusions} we offer some final remarks and point
to outstanding problems related to the Laplace-Beltrami problem and
future avenues of research.

\section{The Laplace-Beltrami problem}
\label{sec:analytical}

In this section we lay out some well-known invertibility
results for the Laplace-Beltrami problem on smooth surfaces, and then
consolidate some existing results for Lipschitz domains.
From the context, the assumptions on~$\Gamma$ should
be clear (i.e. either smooth, general Lipschitz, or piecewise smooth).

\subsection{Smooth surfaces}

To begin with, let~$\Omega \subset \bbR^3$ denote an open bounded
domain whose boundary is given by~$\Gamma$. The boundary~$\Gamma$ can
either be simply or multiply connected, and for now is assumed to be
globally smooth. We shall start by recalling the definition of the
Laplace-Beltrami operator and the Laplace-Beltrami problem on a smooth
surface.

\begin{definition}[Laplace-Beltrami Operator]\label{def:LBO}
If $\Gamma$ is a smooth bounded surface, $\bx = \bx(\theta,\phi)$ is a
local parameterization of~$\Gamma$ with respect to some set of
variables~$\theta,\phi$, the functions $\bs x_\theta$ and $\bs x_\phi$ are the
partial derivatives of $\bs x$, and~$g$ is the associated metric
tensor,
\begin{equation}\label{eq:metric}
  g(\theta,\phi) = \begin{bmatrix}
    \bx_\theta \cdot \bx_\theta & \bx_\theta \cdot \bx_\phi \\
    \bx_\phi \cdot \bx_\theta & \bx_\phi \cdot \bx_\phi
    \end{bmatrix},
\end{equation}
then the Laplace-Beltrami operator~$\LB$ acting on a smooth
function~$f = f(\theta,\phi)$ is defined by the formula
\begin{equation}\label{eq:LB}
  \begin{aligned}
    \LB f  
    &=  \frac{1}{\sqrt{\det g}}
    \begin{bmatrix}
      \partial_\theta & \partial_\phi
    \end{bmatrix}  \sqrt{\det g} \, g^{-1}
    \begin{bmatrix}
      \partial_\theta\\
      \partial_\phi
    \end{bmatrix}  f \\
    &= \surfdiv \surfgrad f.
  \end{aligned}
\end{equation}
\end{definition}
In the above definition, we have introduced the notation~$\surfdiv$
and~$\surfgrad$ to denote the surface divergence and surface gradient,
respectively. The precise definitions in the case of interest, mainly
along piecewise smooth surfaces, are given below in
Section~\ref{sec:lipschitz}.

With the above definition of~$\surflap$, the Laplace-Beltrami problem
refers to solving the equation
\begin{equation} \label{eq:lapbel} 
  \surflap u = f \qquad \text{on } \Gamma
\end{equation}
for the unknown function~$u$.  One must define the domains of the
data~$f$ and solution~$u$ to 
this problem carefully in order to ensure that the problem is
well-posed, as the Laplace-Beltrami operator is neither injective nor
surjective. The operator has a well-known one-dimensional null space:
the space of constant functions along~$\Gamma$~\cite{nedelec2001}. Since the Laplace-Beltrami operator is self-adjoint (see below), the range of the operator is thus the space of functions with zero mean.
A standard  well-posed version of this
problem is summarized in the following theorem, which is proved
in~\cite{Warner2013}: 
\begin{theorem}
  Suppose that $\Gamma$ is a closed surface that is $C^k$ for some
  $k\geq 2$. Then, the Laplace-Beltrami problem~\eqref{eq:lapbel} has a
  unique mean-zero solution for every mean-zero right hand
  side that is in~$L^2(\Gamma)$. Furthermore, if the right hand side is in $H^{s}(\Gamma)$ for some
  $s\leq k-2$, then the solution~$u$ will be in~$H^{s+2}(\Gamma)$.
\end{theorem}

We now move onto a practical discussion of the Laplace-Beltrami
problem along Lipschitz surfaces, a topic which has received less attention in the literature.

\subsection{Lipschitz surfaces}
\label{sec:lipschitz}

In this section, we shall recall the definitions required to clearly
state the Laplace-Beltrami problem on a general Lipschitz
surface~$\Gamma$, as well as summarize the associated invertibility
result presented in~\cite{Gesztesy2011}. Equivalent definitions of the relevant
functions spaces and operators are also discussed
in~\cite{Gesztesy2011} in greater detail. The first notion we outline
is that of a Lipschitz surface. A similar definition to that which we
give below is also contained in~\cite{Mitrea1995}.

\begin{definition}[Lipschitz surface]\label{def:lip_surf}
  Let $\Gamma$ be a bounded surface embedded in $\bbR^3$. The surface~$\Gamma$
  is said to be a Lipschitz surface if there exists a finite open
  covering $\{\cO_j\}_{j=1}^N$ and an associated set of rigid
  rotations~$\Sigma_j$ such that $\Sigma_k(\Gamma\cap \cO_j)$ is the
  graph of a Lipschitz function. That is to say, for each~$j$ there
  exists an open domain~$U_j\subset \bbR^2$ and a Lipschitz
  function~$\phi_j:U_j\to\bbR^3$ such that
  {$\Sigma_j(\Gamma\cap \cO_J)=\{(x,y,\phi_j(x,y))\;|\;(x,y)\in U_j
    \}$}.
\end{definition}
In the above definition, a Lipschitz surface is written in terms of
graphs of Lipschitz functions. It is also possible to describe a Lipschitz surface through local Lipschitz
parameterizations. This description will make it simpler to define the
relevant function spaces and differential operators. With this in mind, we define Sobolev spaces along~$\Gamma$ as
follows.

\begin{definition}[Sobolev spaces]\label{def:sobelev_lip}
  Let $\Gamma$ be a bounded Lipschitz surface with a finite open
  covering $\{\cO_j\}_{j=1}^N$. Also, let $\{\bs x_j\}_{j=1}^N$ be a
  collection of local Lipschitz parameterizations, such that each
  $\bs x_i$ maps an open neighborhood $U_i \subset \bbR^2$ onto
  $\cO_i$. Finally, let $\{\chi_j\}_{j=1}^N$ be a partition of unity
  on~$\Gamma$ such that for each $j$, $\chi_j$ is supported on
  $\cO_j$. We then have the following definitions:
\begin{enumerate}
\item For all $0\leq s \leq 1$ we define the Sobolev space
  along~$\Gamma$ of order $s$ as
\begin{equation*}
  \Hs = \{ f\in L^2(\Gamma)\; | \; (\chi_j f)\circ \bs x_j \in
  H^{s}(U_j) \text{ for all } j= 1,\ldots,N \},
\end{equation*}
with a norm given by
\begin{equation*}
  \left\Vert f \right\Vert_{\Hs} = \sum_{j=1}^N \left\Vert
  (\chi_j f)\circ \bs x_j \right\Vert_{H^{s}(U_j)}.
\end{equation*}

\item For all $-1\leq s<0$, we define the Sobolev space of order $s$,
  $\Hs$, as the dual space of~$H^{-s}(\Gamma)$.

\item For all $-1\leq s\leq 1$, we define $\Hso$ to be the subset of
  $\Hs$ with mean-zero. More explicitly, we set
  \begin{equation}
    \Hso=\left\{ f\in \Hs \;|\;\langle f,1\rangle_{s,\Gamma}=0 \right\},
  \end{equation}
  where~$\langle\cdot,\cdot\rangle_{s,\Gamma}$ is the duality pairing
  between~$\Hs$ and~$H^{-s}(\Gamma)$.

\item We define the space of tangent
  vector fields~$\Ht$ as the subset of three dimensional vector fields
  that are tangential to~$\Gamma$. If~$\bs n$ is the unit normal for~$\Gamma$, defined almost everywhere, then this can be written as
  \begin{equation}
     \Ht = \left\{ \bs v\in (L^2(\Gamma))^3 \; \left | \; \int_\Gamma (\bs v \cdot \bs n)^2 =0 \right. \right\}.
  \end{equation}
\end{enumerate}

\end{definition}

It should be noted that the above definitions are almost exactly those
that are used when $\Gamma$ is smooth. The only difference is that the
$\bs x_i$'s are assumed to be Lipschitz instead of smooth, and we have
restricted~$s$ to~$[-1,1]$. These spaces thus coincide with the usual
Sobolev spaces whenever~$\Gamma$ happens to be smooth. We also note
that for general Lipschitz surfaces, these Sobolev spaces only exist for~$|s|\leq 1$; a local Lipschitz parameterization will
only have fractional derivatives of order~$s$ for~$s\leq
1$. Lastly, we make the usual
identification~$H^{0}(\Gamma)=L^2(\Gamma)$ since~$\|\cdot
\|_{H^{0}(\Gamma)}$ is equivalent to~$\|\cdot\|_{L^2(\Gamma)}$
\cite{Gesztesy2011}.

With the above function spaces defined, we may now move onto
defining the associated surface differential operators in the usual
ways.
\begin{definition}
  \label{def:ders}
  Let $\{\bs x_j\}_{j=1}^N$, and $\{\chi_j\}_{j=1}^N$ be as in Definition~\ref{def:sobelev_lip} and let $g_j$ be the metric associated with $\bx_j$. Interpreting all partial
  derivatives in the distributional sense~\cite{folland1996}, we make the
  following definitions:
\begin{enumerate}
\item For~$0\leq s\leq 1$, the surface gradient $\gradg$ of a function
  $f\in \Hs$ is given by the formula
\begin{equation}
  \gradg f=\sum_j \chi_j\begin{bmatrix}\bx_{j,\theta}& \bx_{j,\phi}\end{bmatrix}g_j^{-1}\begin{bmatrix}
    \partial_\theta\\
    \partial_\phi
    \end{bmatrix}f ,\label{eq:gradg}
\end{equation}
where $\bs x_{j,\theta}$ and $\bs x_{j,\phi}$ are the weak partial derivatives of $\bx_j$ and $g_j$ is the associated metric tensor (see
Definition~\ref{def:LBO}).  It is clear
from this formula that $\nabla_\Gamma:\Ho\to\Ht$.
 
\item The surface divergence~$\divg:\Ht \to H^{-1}(\Gamma)$ is defined as the negative adjoint of the surface gradient map from $\Ho$ to $\Ht$. This may be expressed through the duality product of~$\Hm$ and~$\Ho$:
  \begin{equation}
      ( \divg \bs v, f)_{\Hm,\Ho} := - \int_\Gamma \bs v \cdot  \gradg f,\quad\forall f\in \Ho.
  \end{equation}

\item The Laplace-Beltrami operator~$\LB:\Ho\to \Hm$ is defined as
  the composition~$\LB := \divg \gradg$. The formula for applying
  $\LB$ to a function $f\in\Ho$ is the same as for the smooth
  case~\eqref{eq:LB}, except that the derivatives must be interpreted
  in a weak sense.
\end{enumerate}
\end{definition}


From a practical point of view, the general definition of the surface
divergence given above is difficult to implement numerically. This
obstacle is most easily overcome by assuming a particular (piecewise)
parameterization of~$f$ along~$\Gamma$ and using the following
equivalent definition (see \cite{nedelec2001}): if~$\bs v\in\bs L^{2}_t(\Gamma)$ is written as~$\left.\bs v\right|_{\Gamma\cap \cO_j} =
v_\theta \bs x_{j,\theta}+v_\phi\bs x_{j,\phi}$, then
\begin{equation} \label{eq:divg}
  \left.\divg \bs v\right|_{\Gamma\cap \cO_j}=\frac{1}{\sqrt{\det g}}
  \begin{bmatrix}
    \partial_\theta& \partial_\phi
  \end{bmatrix}\sqrt{\det g}\begin{bmatrix}
    v_\theta\\
    v_\phi   
  \end{bmatrix}.
\end{equation}
 While the above formula still uses weak
derivatives, we shall be able to reinterpret it in a stronger
sense once we make further assumptions on~$\Gamma$.

We also note that the Laplace-Beltrami operator can only be defined on~$\Hs$
for~$s=1$. For any other value of~$s$ the domain of either the surface gradient
or the surface divergence would not be well-defined on a general Lipschitz
surface. We will therefore only consider right hand sides that are contained
in~$\Hm$ when attempting to solve the Laplace-Beltrami equation~$\LB u=f$ on
such surfaces. As in the smooth case, it will be necessary to assume that $f$ is
mean-zero, and we will also look for a mean-zero solution. The Laplace-Beltrami
problem along a Lipschitz surface therefore becomes:
 \begin{problem}[Laplace-Beltrami problem]
   Let $f\in\Hmo$, the set of functions in~$H^{-1}(\Gamma)$ which are also
   mean zero. The Laplace-Beltrami problem is to find a $u\in\Hoo$ such that
   $\LB u=f$.
 \end{problem}
 The
 following theorem shows that the problem is
 well-posed. A proof can be found in~\cite{Gesztesy2011}.
\begin{theorem}
  If $\Gamma$ is a bounded Lipschitz surface without boundary, then the
  null-space of the Laplace-Beltrami operator is the set of constant functions
  and the range is~$\Hmo$. The Laplace-Beltrami problem on $\Gamma$ is thus
  well-posed. \label{thm:LB-dom-rang}
\end{theorem}

In summary, we now have that the distributional form of our equation is
well-posed on a general Lipschitz surface. If we also assume
that~$f\in L^2_{mz}(\Gamma):=\{h\in L^2(\Gamma)|\int_\Gamma h=0\}$, then we may
reformulate this problem more explicitly in the following weak form.
\begin{equation}\label{eq:weak_form}
  \begin{aligned}
    \text{Find } &u\in\Hoo \\
    \text{such that } &\int_\Gamma \nabla_\Gamma u\cdot\nabla_\Gamma v=-\int_\Gamma fv
    \quad  \text{for all } v\in \Hoo.
  \end{aligned}
\end{equation}
In the next section, we will convert the expression above into one with
interface conditions along edges in piecewise smooth geometries.

\section{Interface conditions}
\label{sec:strong}
We shall now further restrict ourselves to the case that the
surface~$\Gamma$ is a finite union of smooth and closed
faces~$\Gamma_i$ and that~$f\in L^2_{mz}(\Gamma)$. We will use this added
smoothness of the geometry to turn the weak
equation~\eqref{eq:weak_form} into a strong interface form on each face
augmented with matching conditions along the edges (but away from any
corners). We start by defining the class of surfaces that we consider.

\begin{definition}[Piecewise smooth Lipschitz surface]\label{def:piecewise_smth}
  Let $\{\Gamma_i\}_{i=1}^N$ be a collection of smooth surfaces such that for each~$\Gamma_i$, there exists a closed triangle~$T_i\subset \mathbb{R}^2$ and a parameterization~$\bs x_i:T_i\to \Gamma_i$ that is a smooth diffeomorphism. We suppose that each~$\Gamma_i$ is closed, so includes its boundary and that the interiors of the faces, denoted~$\left\{\Gamma_i^\interior\right\}$, are pairwise disjoint. If the union~$\Gamma=\cup_{i=1}^N\Gamma_i$ is a Lipschitz surface, then the surface~$\Gamma$
  is said to be a piecewise smooth Lipschitz surface with faces~$\{\Gamma_i\}_{i=1}^N$.
  \end{definition}

\begin{remark}
The above definition may seem restrictive because it requires that every face is a curved triangle. This is not a true restriction as other piecewise smooth surfaces, such as cubes and hemispheres, may be reduced to this case by artificially splitting their faces.

By making this assumption, we can simplify our description of the surface. This in turn will simplify our proofs later on, particularly in Section~\ref{sec:unif_elip}. Similarly, we may ensure that each vertex of~$\Gamma$ is contained in at least three faces by splitting faces as necessary. This assumption will also simplify our proof of Theorem~\ref{thm:acceptable_param}.

It is important to note that not every surface that is smooth almost everywhere can be made to satisfy Definition~\ref{def:piecewise_smth}. For example, cones and other surfaces with conic singularities have unbounded mean curvature, and so are not-piecewise smooth. We will treat the special case of a flat cone in
Section~\ref{sec:cone}.
\end{remark}

We now develop some notation for piecewise smooth Lipschitz surfaces. If it exists, the edge which is the intersection of the faces~$\Gamma_i$ and~$\Gamma_j$
will be denoted by $e_{ij}$. We let $\bs n_i$ be the outward normal
along~$\Gamma_i$, $\bs \tau_i$ be the 
tangent vector along the boundary~$\partial\Gamma_i$ to
face~$\Gamma_i$, and~$\binormal=\bs \tau_i\times \bs n_i$. We will refer
to~$\binormal$ as the binormal vector, and assume that the
orientation of~$\bs \tau_i$ was
chosen so that~$\binormal$ points away
from~$\Gamma_i$. These definitions are demonstrated in Figure~\ref{fig:edge_picture}. We refer to the set of vertices of~$\Gamma$, i.e. the points where some~$\partial\Gamma_i$ is not smooth, as~$\mathcal{C}$. 
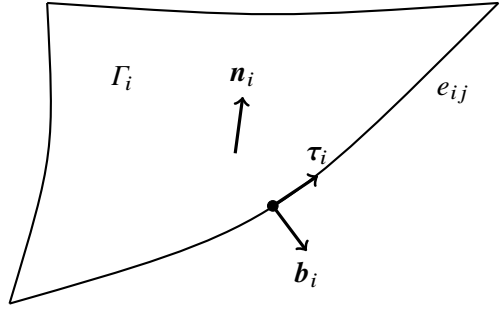
\begin{figure}
    \centering
    \begin{tikzpicture}
        \coordinate (A) at (-3,3);
        \coordinate (B) at (-3.5,-1);
        \coordinate (C) at (3,3);

        \coordinate (N) at (0,0);
        \coordinate (N') at ($ (N)+(0.,0.3) $);
        \coordinate (T) at (-0.5,1);
        
        \coordinate (A') at (3.5,-1);

        \draw[black,thick] (A) .. controls (-2.9,1) .. (B);
        \draw[black,thick] (B) .. controls (N) ..(C);
        \draw[black,thick] (C) .. controls (0,2.8)..(A);

        
        \filldraw[black] (N') circle (2pt);
        \draw[->,very thick] (N') -- ($ (N')+(0.45,-0.6) $) node[below]{$\bs b_i$};
        \draw[->,very thick] (N') -- ($ (N')+(0.6,0.4) $) node[above]{$\bs \tau_i$};

        \draw[->,very thick] (T) -- ($ (T)+(0.1,0.75) $) node[above]{$\bs n_i$};
        \node(label) at (-2,2) {$\Gamma_i$};
        \node(label) at (2.4,1.8) {$e_{i{j}}$};
    \end{tikzpicture}
    \caption{This figure shows some of the definitions associated with our piecewise smooth geometry.}
    \label{fig:edge_picture}
\end{figure}

Using this notation, we compute the interface form of the equation on each
face using integration by parts. For now, we will assume that~$u$ happens to be continuous and in~$C^\infty(\Gamma_i)$ for each~$i$, allowing straightforward integration by parts. After computing
the strong interface form, we will go back and justify that the interface form has a solution that is also a weak solution. For this calculation, we also only consider test functions~$v$ in~\eqref{eq:weak_form} that are continuous. Under these new assumptions, the weak form~\eqref{eq:weak_form}
becomes
\begin{equation}
    \sum_i \left(\int_{\Gamma_i} v\Delta_\Gamma u
    +\int_{\partial\Gamma_i}v \binormal\cdot \nabla_\Gamma u
    \right)=\sum_i\int_{\Gamma_i} fv, \qquad \text{for all } v\in C(\Gamma)\cap \Ho. \label{eq:IBP}
\end{equation}
Through the usual variational arguments~\cite{John1982}, this equation
tells us that~$\LBi u|_{\Gamma_i^{\mathrm{o}}}=f|_{\Gamma_i^{\mathrm{o}}}$ almost everywhere for
each~$i$. By considering choices of~$v$ that are supported near~$e_{ij}$
for each~$i$ and~$j$, equation
\eqref{eq:IBP} also tell us that
\begin{equation}
 \int_{e_{ij}} v \, \left( \binormal\cdot \nabla_\Gamma u|_{\Gamma_i}+\binormalj\cdot
 \nabla_\Gamma u|_{\Gamma_j} \right) = 0 \qquad \text{for all } v\in C(\Gamma)\cap \Ho.
\end{equation}
It is not hard to see that this implies that the one-sided binormal derivatives
of~$u$ agrees from both sides of an edge. We thus have that if the
solution $u$ happens to be smooth enough that we can do the above integration by parts, then~$u$ solves the following problem:
\begin{problem}[Interface form of the Laplace-Beltrami problem]
  \label{def:interface_form}
 Let $\Gamma$ be a closed Lipschitz surface composed of smooth faces $\left\{\Gamma_i\right\}$ and
 suppose that $f\in L_{mz}^2(\Gamma)$. The interface form of the
 Laplace-Beltrami problem is to find a $u\in H^{1}_{mz}(\Gamma)$ such that~$\frac{\partial u|_{\Gamma_i}}{\partial \binormal}$ exists in the trace sense defined below and such that
 \begin{equation}
  \begin{aligned}
    \LB u|_{\Gamma_i^{\mathrm{o}}} &=f|_{\Gamma_i^{\mathrm{o}}},  &\qquad &\text{a.e. on } \Gamma_i,\\
    \frac{\partial u|_{\Gamma_i}}{\partial \binormal} &=-\frac{\partial
      u|_{\Gamma_j}}{\partial \binormalj} & &\text{on } e_{ij},\\
    u|_{\Gamma_i}&=u|_{\Gamma_j} & &\text{on } e_{ij} ,
  \end{aligned}\label{eq:interface_form}
\end{equation}
where the edge conditions are interpreted in the trace sense.
\end{problem}
\noindent
We note here that the interface form does not involve any corner conditions
on~$u$ beyond the requirement that the traces exist. This comes from the fact
that integration by parts does not introduce any corner conditions on piecewise
smooth domains.
   
Having identified the interface form of the Laplace-Beltrami equation, we will
go back and prove that it is well posed and equivalent to the weak form
whenever~$f$ is in~$L^2_{mz}(\Gamma)$. We do this in four stages. First, in
Section~\ref{sec:planar_results}, we recall standard results for elliptic
interface problems in the plane. Then, in Section~\ref{sec:unif_elip} we connect
the Laplace-Beltrami problem on a surface patch to an interface problem in the
plane and we show that there exists a covering by parameterizations such that
the corresponding elliptic interface problem is uniformly elliptic.
Subsequently, in Section~\ref{sec:interface_form_hols} we check that the
remaining requirements of the planar regularity theorems
(Theorems~\ref{thm:Nicaise} and~\ref{thm:Nicaise2}) are satisfied so that we can
therefore rigorously carry out the above integration by parts argument on a
single surface patch (Theorem~\ref{thm:single_patch_result}). Finally, in
Section~\ref{sec:closed_surface} we use these local results to prove that the
weak solution of the Laplace-Beltrami problem satisfies the interface form on
the whole surface (Theorem~\ref{thm:strongLB}).

\subsection{Interface problems in the plane}  \label{sec:planar_results}
In order to prove the regularity of weak solutions of the Laplace-Beltrami problem, we shall leverage existing results for elliptic interface problems in the plane. To start with, with define such problems.
\begin{definition}[Elliptic interface problem in the plane]\label{def:ellip_int}
Let the following hold:
\begin{enumerate}
    \item The set $\Omega = \cup_{i=1}^N T_i\subset\bbR^2$ is a finite union of triangles such that every pair of triangles $T_i$ and $T_j$ either overlap along an entire edge, overlap only at a vertex shared by both triangles, or don't overlap.
    \item The vector field $\bs \nu_i$ is the outward unit normal to the boundary of~$T_i$ for~$i=1,\ldots,N$, defined almost everywhere on that boundary.
    \item The differential operators $L_i$ are elliptic on~$T_i$ for~$i=1,\ldots,N$ and can be written as $L_i = \operatorname{trace}[g^{-1}_i \nabla^2 ] + \bs h_i \cdot \nabla $, where~$g_i^{-1}$ is a matrix-valued function and~$\bs h_i$ is a vector field defined on $T_i$.
    \item The function $f$ is in $L^2(\Omega)$.
    \item The functions $\alpha_i$ are positive functions on $\partial T_i$
    \item The functions $c_i$ are in~$L^2(\partial T_i \setminus \cup_{j\neq i} \partial T_j)$ for~$i=1,\ldots,N$.
\end{enumerate}
The elliptic interface problem is to find a function~$u$ such that
\begin{equation}
\begin{aligned}
    L_i u|_{T_i}&= \left.f\right|_{T_i} &\qquad &\text{a.e. on } T_{i},\\
    u|_{T_i} &= u|_{T_j} &\qquad &\text{a.e. on } T_{i} \cap T_{j},\\
   \alpha_i (g^{-1}_i  \bs \nu_{i})\cdot \nabla u|_{T_i} &= - \alpha_j(g^{-1}_j \bs \nu_j)\cdot \nabla u|_{T_j} &\qquad &\text{a.e. on } T_{i} \cap T_{j},\\
   u|_{T_i} &= c_i&\qquad &\text{a.e. on } \partial T_i\cap \partial\Omega,
\end{aligned}\label{eq:elliptic_in_plane}
\end{equation}
for all~$i$ and~$j$ such the corresponding sets are non-empty, that~$u$ is in $H^1(T_i)$ for each $i$, and that the trace of $\nabla u$ can be taken in then sense discussed below Corollary~\ref{cor:cornerstodomain}.
\end{definition}
The regularity of solutions to an elliptic interface problem will depend
strongly on their behaviour near vertices. To analyze this, we look at
homogeneous solutions of~\eqref{eq:elliptic_in_plane} in the region around a
single interior or exterior vertex. These corner solutions are discussed in detail
in~\cite{Nicaise1994}. Here we summarize by noting that they are of the form
$r_s^{\lambda_s} \tau_s(\theta_s)$, where $(r_s,\theta_s)$ are polar coordinated
centered at the vertex $s$. The class of allowable $\lambda_s$ and $\tau_s$ are
defined below. For some higher order problems, there also exist corner solutions
which behave
as~$r_s^{\lambda_s} \lp\ln r_s\rp^q \tau_s(\theta_s)$, where $q$ is some positive
integer, but those do not occur in our second-order problem (see
Lemma~\ref{lem:corners} below), so we neglect them in this discussion.
\begin{definition}
Let $s\in \bbR^2$ be a vertex in an elliptic interface problem and let $(r_s,\theta_s)$ be the polar coordinates centered at $s$. The complex number $\lambda_s$ and continuous, $2\pi$-periodic, and piecewise smooth function $\tau_s$ form an expansion pair for the vertex $s$ in the elliptic interface problem if they satisfy the differential condition
\begin{equation}
   L_i r_s^{\lambda_{s}}
        \tau_s(\theta_s)=o(r^{\lambda_s-2}) \qquad \text{as } r_s\to 0
\end{equation}
 for all $i$ such that $T_i$ touches $s$, the matching condition
\begin{equation}
    \alpha_i(g^{-1}_i \bs \nu_{i})\cdot \nabla \lp r_s^{\lambda_{s}}\tau_s\rp|_{T_i} = - \alpha_j(g^{-1}_j \bs \nu_j)\cdot \nabla \lp r_s^{\lambda_{s}}\tau_s(\theta_s)\rp|_{T_j} +o(r^{\lambda_s-1}) \qquad \text{as } r_s\to 0
\end{equation}
for all $i$ and $j$ such that $T_i$ and $T_j$ share an edge touching $s$, and the boundary condition~$\tau_s=0$ on~$\partial\Omega$. In the following, we refer to $\lambda_s$ as an expansion power for the vertex $s$ and to $\tau_s$ as the angular function associated with $\lambda_s$.
\end{definition}

For the above expansion powers to be discrete, we must impose some more assumptions on the operators in the interface problem. We now
define an alternative, more regular, interface problem.
\begin{definition}[Regular elliptic interface problem]
  An elliptic interface problem is regular if the matrix valued functions
  $g_i^{-1}$ are uniformly positive definite on their domains and the
  coefficient functions $g_i^{-1}$,~$h_i$,~and~$\alpha_i$ are smooth and
  bounded.
\end{definition}

The following two theorems give the existence and regularity of the solution of
an elliptic interface problem. These are Theorems 4.2 in~\cite{Nicaise1994} and
8.6 in~\cite{Nicaise1994a}, restricted to the case of second-order equations
with Dirichlet boundary conditions on~$\partial\Omega$.

\begin{theorem}[Existence, Nicaise and S\"andig Theorem 4.2]\label{thm:Nicaise}
Let~$\mathcal{V}$ be the set of corners of the domain~$\Omega$ in Definition~\ref{def:ellip_int}, i.e. points that are a vertex of some~$T_i$. Also let~$\{\lambda_{n,s}\}_n$ be the expansion powers of the vertex~$s\in \mathcal{V}$ for an elliptic interface problem. If all of the powers~$\lambda_{n,s}$ are real and the right hand side~$f$ is in~$L^2(\Omega)$, then there exists a unique function $v\in
H^1(T^{\mathrm{o}})$ that solves that elliptic interface problem in a sense described in Section~3 of~\cite{Nicaise1994} with~$c_i=0$ for each~$i$. 
\end{theorem}

\begin{theorem}[Regularity, Nicaise and S\"andig Theorem 8.6]\label{thm:Nicaise2}
  Suppose that the assumptions of Theorem~\ref{thm:Nicaise} hold and $v$ is the
  identified solution. If $f$ is also in~$H^{k}(T_i^{\mathrm{o}})$ for every
  $i$, where $k=0,1,2$, and the neighborhood~$V$ contains exactly one vertex
  of~$\Omega$, then for any point in~$V$ we may write the solution~$v$ as
  \begin{equation}
      v=v_0+\sum_{\substack{\lambda_{n,s}\leq
         (k+1)\\\lambda_{n,s}\neq 0}} a_{n} \, r_s^{\lambda_{n,s}} \,
        \tau_{n,s}(\theta_s),\label{eq:corner_exp}
  \end{equation}
  where the function~$v_0$ is in~$H^{k+2}(T_i^{\mathrm{o}}\cap V)$ for
  each~$i$ such that~$T_i$ touches~$s$ and satisfies the interface conditions in~\eqref{eq:elliptic_in_plane} in the usual trace sense, where~$a_n\in\mathbb{C}$ for each~$n$, where $(r_s,\theta_s)$ are polar coordinates around the corner $s$, and where $\tau_{n,s}$ is the expansion function associated with $\lambda_{n,s}$.
\end{theorem}
\begin{remark}
  Nicaise and S\"andig Theorem 8.6 also includes the case where $k>2$, but when
  looking for higher regularity, equation~\eqref{eq:ucorner_hom} must also
  include weak interface singularities. We thus restricted
  Theorem~\ref{thm:Nicaise2} to the case of moderate regularity ($k\leq 2$) in
  order to simplify our discussion.
\end{remark}
The above theorem can be applied to give a global representation of the solution $v$. This extension will need the partition of unity given in the following lemma to ensure that each part of the decomposition satisfies the interface conditions.
\begin{lemma}\label{lem:part_of_unity}
  If~$\{V_s\}_{s\in\mathcal{V}}$ is a collection of neighborhoods covering~$\Omega$ such that each $V_s$ contains only the vertex $s$, then there exists a partition of unity~$\{\zeta_s\}$ such that each~$\zeta_s$ is supported in $V_s$, is smooth in each region~$T_i$, and satisfies the interface conditions in~\eqref{eq:elliptic_in_plane}.
\end{lemma}
\begin{proof}

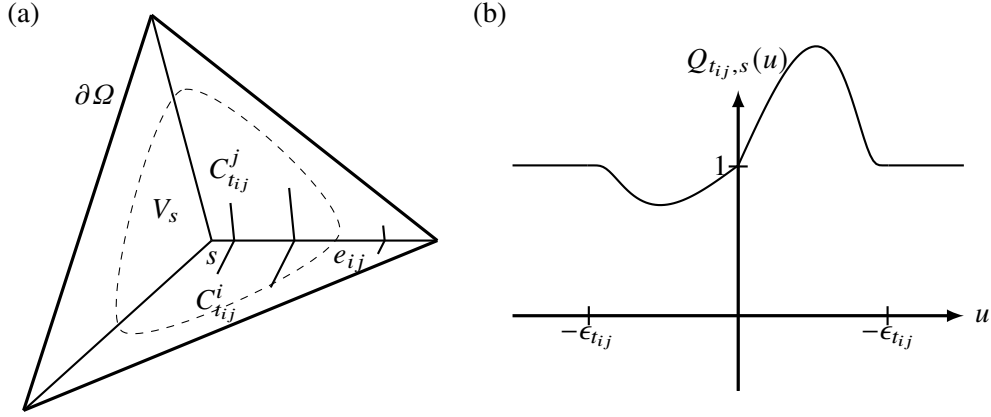
\begin{figure}[t]
    \centering
    \begin{tikzpicture}
\coordinate (A) at (3,0);
\coordinate (A') at (1.7,0.1);
\coordinate (C) at (-0.8,3);
\coordinate (C') at (-0.4,2);
\coordinate (D) at (-2.5,-2.26);
\coordinate (D') at (-1.2,-1.2);

\coordinate (O) at (0,0);
    
\draw[black, thick] (O) node [below] {$s$} -- (A);
\draw[black, thick] (O) -- (C);
\draw[black, thick] (O) -- (D);

\draw[black, very thick] (C)-- (A);
\draw[black, very thick] (D) -- (A);
\draw[black, very thick] (D) -- (C) node[pos=0.8,left] {$\partial\Omega$};

\draw[dashed] plot [smooth cycle] coordinates {(A') (C') (D')};
\node(label) at (-0.6,0.4) {$V_s$};


\coordinate (E) at (0.3,0);
\coordinate (F) at (1.1,0);
\coordinate (G) at (2.3,0);

\coordinate (b1) at (-0.45,-0.89);
\coordinate (b2) at (-0.1,1);

\draw [thick] (E) -- ($(E)+0.5*(b1)$) node [below] {$C^i_{t_{ij}}$};
\draw [thick] (F) -- ($(F)+0.7*(b1)$);
\draw [thick] (G) -- ($(G)+0.2*(b1)$);

\draw [thick] (E) -- ($(E)+0.5*(b2)$) node [above] {$C^j_{t_{ij}}$};
\draw [thick] (F) -- ($(F)+0.7*(b2)$);
\draw [thick] (G) -- ($(G)+0.2*(b2)$);


\draw[black,thick,-|] (7,-2) -- (7,1) node [left] {$1$};
\draw[black,thick,|-|] (5,-1) node [below] {$-\epsilon_{t_{ij}}$} -- (9,-1) node [below] {$-\epsilon_{t_{ij}}$};
\draw[black,very thick,-latex] (7,-2) -- (7,2) node [above] {$Q_{t_{ij},s}(u)$};
\draw[black,very thick,-latex] (4,-1) -- (10,-1) node [right] {$u$};

\draw[black,thick] (4,1) -- (5,1);
\draw[black,thick] (10,1) -- (9,1);
\draw [thick,samples=50,domain=0:1] plot ({-2*\x+7}, {1-4*\x*exp(-1/(1-\x^2))});
\draw [thick,samples=50,domain=0:1] plot ({2*\x+7}, {1+12*\x*exp(-1/(1-\x^2))});

\node(label) [below] at (1.84,0) {$e_{ij}$};

\node(label) at (-2.5,3) {(a)};
\node(label) at (3.7,3) {(b)};
\end{tikzpicture}
\caption[A partition of unity that satisfies interface conditions]{(a) An example domain $\Omega$ split into three regions. We consider the vertex $s$ in the center and the neighborhood $V_s$ containing $s$ and no other vertex of $\Omega$. Along the edge $e_{ij}$, we define the line segments $C^{i}_{t_{ij}}$ that will become the support of the correction function $Q_{t_{ij},s}$.
    (b) An example of a correction function $Q_{t_{ij},s}$. By definition it is continuous, positive, and identically one outside of the region $\left(-\epsilon_{t_{ij}},\epsilon_{t_{ij}}\right)$.}
    \label{fig:ang_corr}
\end{figure}




    








Let $\{\eta_s\}_{s\in\mathcal{V}}$ be a smooth partition of unity covering
$\Omega$ and such that $\eta_s$ is compactly supported in $V_s$. We shall modify this
partition of unity to satisfy the interface conditions. We suppose that~$t_{ij}$ is an arc-length parameter for the edge $e_{ij}=T_i\cap T_j$. For each value of~$t_{ij}$, we define two line segments~$C^i_{t_{ij}}$ and~$C^j_{t_{ij}}$ that start on~$e_{ij}$, are tangent to~$g_i^{-1}\bs \nu_i$ and~$g_j^{-1}\bs \nu_j$ respectively, and have length $\epsilon_{t_{ij}}>0$ (see Figure~\ref{fig:ang_corr}). We set $\epsilon_{t_{ij}}$ such that~$C^i_{ij}$ and~$C^j_{ij}$ don't include any other edge $e_{i'j'}$ or
$\partial\Omega$. We also assume that
$\epsilon_{t_{ij}}$ is chosen to vary smoothly with $t_{ij}$ and is less than the maximum of~$e|g^{-1}_i\nabla \eta_s|$ and~$e|g^{-1}_i\nabla \eta_s|$ for all~$s$.

On each line segment, we define a piecewise smooth correction
function $Q_{t_{ij},s}$ as
\begin{equation}
    Q_{t_{ij},s}(u) = 1+\begin{cases}
        -e\left.g^{-1}_i\bs \nu_i \cdot \nabla \eta_s\right|_{e_{ij}} u \exp\lp -\frac{1}{1-(u/\epsilon_{t_{ij}})^2}\rp & -\epsilon_{t_{ij}}<u<0\\
        -e\left.g^{-1}_j\bs \nu_j \cdot \nabla \eta_s\right|_{e_{ij}} u \exp\lp -\frac{1}{1-(u/\epsilon_{t_{ij}})^2}\rp & 0\leq u< \epsilon_{t_{ij}}
    \end{cases}
\end{equation}
where~$u$ is a parameter for the union of~$C^i_{t_{ij}}$ and~$C^j_{t_{ij}}$.
With this definition,~$Q_{t_{ij},s}$ can be smoothly extended by one on either side of~$e_{ij}$. The product rule gives that $\left.g^{-1}_i\bs \nu_i \cdot \nabla Q_{t_{ij},s}\eta_s\right|_{e_{ij}}=0$, so $Q_{t_{ij},s}\eta_s$ satisfies the interface conditions on
$e_{ij}$ (see Figure~\ref{fig:ang_corr}b). 

We note that $Q_{t_{ij},s}$
varies smoothly with $t_{ij}$ because $\eta_s$ and $\epsilon_{t_{ij}}$
do. With these assumptions, we have that $Q_{t_{ij}}\eta_s$ is smooth on the
interior of $T_i$. The fact that $\eta_s$ must be flat around either end of
$e_{ij}$ will then give that $Q_{t_{ij},s}\eta_s$ is smooth up to $\partial T_i$
for each $i$, even though~$\epsilon_{t_ij}$ might vanish there. We also note that $Q_{t_{ij},s}$ is positive because of the upper bound on~$\epsilon_{t_{ij}}$. We thus have that $Q_{t_{ij},s}\eta_s$ has the same support as
$\eta_s$.

Including all such corrections, we define
$\tilde \zeta_s = \lp\prod_{ij}Q_{t_{ij},s}\rp\eta_s$. By definition,
$\tilde\zeta_s$ will be non-negative and satisfy the interface conditions in
\eqref{eq:elliptic_in_plane} and be smooth on $T_i$ for each $i$. We also know
that $\sum_{s\in\mathcal{V}}\tilde\zeta_{s}$ is never zero in $\Omega$ because
$\{\eta_s\}$ cover $\Omega$. The functions
$\zeta_s = \tilde \zeta_s / \sum_{s'\in\mathcal{V}}\tilde\zeta_{s'}$ will thus
form the desired partition of unity.
\end{proof}
\begin{corollary}
  \label{cor:cornerstodomain}
  Let~$\{V_s\}_{s\in\mathcal{V}}$ and $\{\zeta_s\}$ be as in Lemma~\ref{lem:part_of_unity}. Then the solution of the elliptic interface problem~$v$ identified in Theorems~\ref{thm:Nicaise} and~\ref{thm:Nicaise2} can be written as
  \begin{equation}
      v=v_0+\sum_{s\in\mathcal{V}}\zeta_s\sum_{\substack{\lambda_{n,s}\leq
         (k+1)\\\lambda_{n,s}\neq 0}} a_{n,s} \, r_s^{\lambda_{n,s}} \,
        \tau_{n,s}(\theta_s),
  \end{equation}
where the function~$v_0$ is in~$H^{k+2}(T_i^{\mathrm{o}})$ for each~$i$ and satisfies the interface conditions.
\end{corollary}
In the above theorem, if~$\lambda_{n,s}<1$, for any~$n$, then the solution~$u$
will not be smooth enough for us to take the trace of~$\nabla u$ in the
traditional sense. In this case, we compute the trace along an edge by
evaluating the explicit function expansion terms at the edge and adding the
result to the trace of $v_0$.

\subsection{Regularity of the pull-back of the Laplace-Beltrami problem}\label{sec:unif_elip}

In order to see the equivalence between the interface form of the
Laplace-Beltrami problem~\eqref{eq:interface_form} on $\Gamma$ to
problem~\eqref{eq:elliptic_in_plane}, we must define the trace operator and a
function space that is smooth enough so that we may take the trace of the
derivative. To define this space, we note that on a piecewise smooth Lipschitz
surface, the spaces $H^s(\Gamma_i^{\mathrm{o}})$ may be defined for each $i$ and
for any $s\in \bbR$ through an analogue of Definition~\ref{def:sobelev_lip} based on smooth parameterizations,
even though $\Hs$ may not exist. We thus define the
space of piecewise $H^s$ functions as follows:
\begin{definition}
For any $s\geq 0$, set
\begin{equation}
    \mathcal{H}^s(\Gamma):= \left\{ v\in L^2(\Gamma):\; v|_{\Gamma_i^{\mathrm{o}}}\in
    H^s(\Gamma_i^{\mathrm{o}}) \text{ for all } i \right \} .
\end{equation}
\end{definition}
Having identified the appropriate function space, we now formally state the
definition of the trace along an edge:
\begin{definition}[Trace]
  Suppose~$\bs x_i$ is a smooth parameterization of~$\Gamma_i$ with domain $T_i$.
  If~$u\in H^s(\Gamma_i^{\mathrm{o}})$ for some~$s>\frac 12$, then the trace
  of~$u$ is defined as
\begin{equation}
    \tr_{\partial \Gamma_i}u:= \tr_{\partial T_i} (u\circ \bs x_i)\circ \bs x_i^{-1},
\end{equation}
where~$\tr_{\partial T}$ is the usual trace operator for the bounded Lipschitz domain~$T$, which is defined in Theorem~18.1 of~\cite{Leoni2017} .

Furthermore, if~$v\in \mathcal{H}^s(\Gamma)$ for some~$s>\frac12$, then the trace of~$v$ along a surface edge is defined as
\begin{equation}
     \tr_{\partial \Gamma_i}v =  \tr_{\partial \Gamma_i} \lp v|_{\Gamma_i^\interior}\rp.
\end{equation}
\end{definition}
By Theorem~18.1 of~\cite{Leoni2017}, we know that the range of the trace operator is~$H^{s-\frac{1}{2}}(\partial\Gamma_i)$, which may be defined equivalently to~$H^{s-\frac{1}{2}}(\Gamma_i)$ in Definition~\ref{def:sobelev_lip}.

\begin{theorem}
    The operator~$\operatorname{tr}_{\partial\Gamma_i}$ is independent of the local parameterization used in its definition.
\end{theorem}
\begin{proof}
We begin by showing that~$C^\infty(\Gamma_i)$ is dense in~$H^s(\Gamma_i^{\mathrm{o}})$ for all~$s>0$. We suppose that~$\bx_i$ is a smooth parameterization of~$\Gamma_i$ with domain~$T_i$. By definition,~$H^s(\Gamma_i^{\mathrm{o}})=\left\{f\in L^2(\Gamma_i)| f\circ\bx_i \in H^s(T_i)\right\}$ and $C^\infty(\Gamma_i)=\left\{f\in C(\Gamma_i)| f\circ\bx_i \in C^\infty(T_i)\right\}$. The standard density result then implies that~$C^\infty(\Gamma_i)$ is dense in~$H^s(\Gamma_i^{\mathrm{o}})$.

Next, we note that if~$v\in C^\infty(\Gamma_i)$,
then~$\tr_{\partial\Gamma_i}v$ is independent of the smooth local parameterization used
in its definition. The density of of~$C^\infty(\Gamma_i)$ thus implies that the trace operator is independent of the local parameterization.
\end{proof}

We now state and prove the equivalence of the interface form of the Laplace-Beltrami problem with Dirichlet boundary conditions on a single surface patch to the above elliptic problem in the following theorem. We shall specifically consider surface patches consisting of several faces: $\tilde \Gamma =\cup_{i\in \mathcal{I}}\Gamma_i$, where $\mathcal{I}$ is a subset of~$\{1,\ldots,N\}$.

\begin{theorem}\label{thm:equivalent_inter}
    Let $\tilde\Gamma$ be a patch of $\Gamma$ parameterized by a piecewise smooth parameterization $\tilde\bx$ with domain $U$. Also suppose that $T_i=\tilde\bx^{-1}(\tilde \Gamma\cap\Gamma_i)$. Finally, let $f\in L^2(\tilde\Gamma)$ and let $c\in L^2(\partial\tilde \Gamma)$. A function~$u$ solves the interface form of the Laplace-Beltrami problem on the $\tilde\Gamma$ with right hand side $f$ and Dirichlet boundary data $c$ if and only if~$u\circ \tilde\bx^{-1}$ solves the elliptic interface problem in the plane with right hand side $\tilde f=f\circ \tilde\bx^{-1}$, derivative matching coefficients $\alpha_i= \left(\bs b_i^Tg^{-1}\bs b_i\right)^{-1/2}$, boundary data $c_i=(c \circ \tilde\bx^{-1})|_{\partial T_i \setminus \cup_{j\neq i} \partial T_j}$, and~$L_i$ being the pull-back from $\Gamma_i$ of the Laplace-Beltrami operator by~$\tilde\bx$~\eqref{eq:LB_expl}.
\end{theorem}
\begin{proof}

We first check that the domain and operators in the plane satisfy the requirements of Definition~\ref{def:ellip_int}. The domain of the piecwise smooth local parameterization $\tilde\bx$ can be decomposed into subdomains $T_i$ corresponding to each face in $\tilde\Gamma$. Each pair of these $T_i$'s either overlap along an edge or just at a vertex because the parameterization $\tilde\bx$ is injective and $\tilde\Gamma$ is a subsurface of a piecewise smooth Lipschitz surface.

Next, we may use the explicit forms of the surface gradient and divergence above, to see that the operators $L_i$ have the form
\begin{equation}  
\begin{aligned}\label{eq:LB_expl}
L_i u &= 
\frac{1}{\sqrt{\det g_i}}
    \begin{bmatrix}
      \partial_\theta & \partial_\phi
    \end{bmatrix}  \sqrt{\det g_i} \, g_i^{-1}
    \begin{bmatrix}
      \partial_\theta\\
      \partial_\phi
    \end{bmatrix}  u \\
    &= \text{trace}\left[g_i^{-1}     \begin{bmatrix}
      \partial_\theta^2&\partial_\theta\partial_\phi\\
      \partial_\theta\partial_\phi&\partial_\phi^2
    \end{bmatrix}  u \right]+\frac{1}{\sqrt{\det g_i}}\lp\begin{bmatrix}
      \partial_\theta & \partial_\phi
    \end{bmatrix}  \sqrt{\det g_i} \, g_i^{-1}\rp \begin{bmatrix}
      \partial_\theta\\
      \partial_\phi
    \end{bmatrix}  u,
    \end{aligned}
\end{equation}
where $g_i$ is the surface metric on $T_i$~\eqref{eq:metric}. Letting
\[
  \bs h_i=\frac{1}{\sqrt{\det g_i}}\lp\begin{bmatrix}
    \partial_\theta & \partial_\phi
  \end{bmatrix}  \sqrt{\det g_i} \, g_i^{-1}\rp,
\]
using the definition of the metric, and the fact that $\tilde\bx$ is piecewise smooth, we
have that $L_i$ is an elliptic operator on each $T_i$, though it may not be
uniformly elliptic.
The Lipschitz property of $\tilde\bx$ will also give that $\tilde f$, $\alpha_i$,
and $c_i$ have the required integrability. We have thus identified the elliptic
interface problem corresponding to the interface form Laplace-Beltrami problem
on $\tilde\Gamma$.

We now check that~\eqref{eq:elliptic_in_plane} is equivalent to
\eqref{eq:interface_form}. The first equation follows directly from the
definition of the pull-back of the Laplace-Beltrami operator. The second and
fourth equations in~\eqref{eq:elliptic_in_plane} are equivalent to the trace
conditions in~\eqref{eq:interface_form} because of our definition of the trace.
To see the equivalence of the derivative matching conditions, we let $\bs t_i$ be
the unit tangent to $\partial T_i$. The vector field
$\begin{bmatrix}\tilde\bx_\theta& \tilde\bx_\phi\end{bmatrix}\bs t_i$ is then tangent to
$\partial \Gamma_i$. We also write the binormal vector $\bs b_i$ to
$\partial\Gamma_i$ as $\begin{bmatrix}\tilde\bx_\theta& \tilde\bx_\phi\end{bmatrix}\bs w_i$.
Combining these gives that
\begin{equation}
  0= \begin{bmatrix}\tilde\bx_\theta& \tilde\bx_\phi\end{bmatrix}\bs t_i\cdot \bs b_i = \begin{bmatrix}\tilde\bx_\theta& \tilde\bx_\phi\end{bmatrix}\bs t_i \cdot \begin{bmatrix}\tilde\bx_\theta& \tilde\bx_\phi\end{bmatrix}\bs w_i = \bs t_i \cdot g_i \bs w_i,
\end{equation}
where $g_i$ is the surface metric~\eqref{eq:metric}. We thus have that $g_i \bs w_i$ is perpendicular to $\bs t_i$. This gives that $\bs w_i =\alpha_i g^{-1}_i \bs \nu_i$, where the value of $\alpha_i$ given above guarantees that $\bs b_i$ is normalized. We may then use the definition of the surface gradient to compute
\begin{equation}
  \begin{aligned}
    \bs b_i \cdot \tr_{\partial\Gamma_i} \gradg u &= \alpha_i \bs \nu_i\cdot \tr_{T_i}(g_i^{-1} g_i \nabla (u\circ\tilde\bx^{-1})) \\
    &= \alpha_i \bs \nu_i\cdot \tr_{T_i} \nabla (u\circ\tilde\bx^{-1}).
  \end{aligned}
\end{equation}
Applying this identity to both sides of the edge $e_{ij}$ will give the equivalence of the derivative matching conditions. We have thus proved that $u\circ \tilde\bx^{-1}$ satisfies each part of~\eqref{eq:elliptic_in_plane} if and only if $u$ satisfies each part of~\eqref{eq:interface_form}.
\end{proof}

With the above theorem, we can see how we might apply Theorems~\ref{thm:Nicaise} and~\ref{thm:Nicaise2} to give results on the Laplace-Beltrami problem. It remains, however, to show that there are parameterizations such that the elliptic interface problem identified in Theorem~\ref{thm:equivalent_inter} is regular. This will be simplest if we specifically consider $\tilde\Gamma$ to be the union of faces touching a surface vertex $S$: $\Gamma_S = \cup_{\{i:S\in\Gamma_i\}}\Gamma_i$. We thus wish to prove the following theorem.

\begin{theorem}[Suitable parameterization]\label{thm:acceptable_param}
  There exists a piecwise smooth parameterization $\bx_S$ of $\Gamma_S$ with domain $U_S$ such that the elliptic interface problems identified in Theorem~\ref{thm:equivalent_inter} is regular.
\end{theorem}
To prove this theorem, we start by noting that the pull-back of the
Laplace-Beltrami operator by a smooth parameterization of a single face is
uniformly elliptic.
\begin{lemma}\label{thm:unif_elip}
Let~$\Gamma_i$ be a closed face of a piecewise smooth Lipschitz surface~$\Gamma$ with a smooth parameterization~$\bx_i$ that maps the triangle $T_i$ to $\Gamma_i$. The pull-back $L_i$ of the Laplace-Beltrami operator under~$\bx_i$,
\begin{equation}
  L_i u  
    :=  \LB (u\circ \bx_i^{-1})\circ \bx_i,\label{eq:pull-back}
    \end{equation}
    is uniformly elliptic with smooth coefficients on~$T_i$, the domain of $\bx_i$. 
\end{lemma}
\begin{proof}
  The coefficient in front of each derivative term in~\eqref{eq:LB_expl} is
  in~$C^\infty(T_i)$ because~$\bx_i$ was assumed to be a smooth diffeomorphism
  on~$T_i$ up to its boundary. To see that the pull-back is uniformly elliptic,
  we note that the partial derivatives of~$\bx_i$ are bounded. The trace of~$g$
  is therefore bounded above. Since $g^{-1}$ is positive definite, this in turn
  ensures that the eigenvalues of~$g^{-1}$ are bounded away from zero. The
  operator $L_i$ is thus uniformly elliptic on the closed triangle~$T_i$.
\end{proof}

The parameterizations in Lemma~\ref{thm:unif_elip} cannot be used to study the
effect of surface edges on the solution of the Laplace-Beltrami problem because
they do no cover the edges. Instead, we use them as building blocks to create the
desired local parameterizations of the patches $\Gamma_S$: we shall combine the
parameterizations of several faces into a single piecewise smooth
parameterization by manipulating the domain of the parameterizations of the
involved faces. The first such manipulation is contained in the following lemma.

\begin{lemma}\label{lem:new_param}
Let~$\Gamma_i$ be a closed face of a piecewise smooth Lipschitz surface~$\Gamma$. Also let~$T_i\subset\bbR^2$ be any closed non-degenerate triangle. There exists a smooth parameterization~$\bx_i$ of~$\Gamma_i$ whose domain is~$T_i$. 
\end{lemma}
\begin{proof}
Let~$\tilde\bx$ be a smooth parameterization mapping some triangle~$\tilde T\subset\bbR^2$ to~$\Gamma_i$, which must exist by definition. Let the vertices of~$\tilde T$ be~$\tilde{\bs v}_1,\tilde{\bs v}_2,$ and~$\tilde{\bs v}_3$. Let the vertices of~$T_i$ be $\bs v_1, \bs v_2,$ and~$\bs v_3$. We define the affine transformation $\chi:\bbR^2\to\bbR^2$ as
\begin{equation*}
    \chi(\bs v) = \begin{bmatrix}
      {\bs v}_2-\bs v_1&{\bs v}_3-\bs v_1
    \end{bmatrix}\begin{bmatrix}
      \tilde{\bs v}_2-\tilde{\bs v}_1&\tilde{\bs v}_3-\tilde{\bs v}_1
    \end{bmatrix}^{-1}(\bs v-\tilde{\bs v}_1)+\bs v_1.
\end{equation*}
It is clear that~$\chi(\tilde T) = T_i$, since if $a_1,a_2$, and $a_3$ are three real numbers that add to 1, then
\begin{equation*}\chi(a_1\tilde{\bs v}_1+a_2\tilde{\bs v}_2+a_3\tilde{\bs v}_3)=a_1\bs v_1+a_2\bs v_2+a_3\bs v_3.
\end{equation*}
The transformation~$\chi$ is also a smooth diffeomorphism because the matrices in its definition are non-singular. The parameterization~$\bx_i=\tilde \bx\circ \chi^{-1}$ is therefore the desired one.
\end{proof}

In order to convert the interface form of the Laplace-Beltrami problem into a regular interface problem in the plane, we shall construct a piecewise smooth parameterization $\bx_S$ that maps its domain onto several adjacent faces. We construct $\bx_S$ piecewise, by using the above lemma to pick parameterizations of several faces such their domains are adjacent triangles in the plane. The above lemma is not sufficient, however, to ensure that $\bx_S$ is continuous: the parameterizations of two adjacent faces need not agree along the edge shared by their faces. To force the paramaterizations to agree, we shall suppose that they preserve the arclength along that edge. We prove this is possible in the following lemma:

\begin{lemma}\label{lem:arc-length}
Let~$\Gamma_i$ be a face of a piecewise smooth surface~$\Gamma$ and $e_{ij_1}$ and~$e_{ij_2}$ be edges of~$\Gamma_i$. There exists a parameterization~$\bx_i$ of~$\Gamma_i$ whose triangular domain~$T$ has edges~$e_1$, $e_2$, and~$e_3$, such that $e_1$ and $e_2$ are mapped onto $e_{ij_1}$ and~$e_{ij_2}$ in an arclength-preserving manner.
\end{lemma}
\begin{proof}
Let~$\tilde\bx$ be a smooth parameterization mapping some triangle~$\tilde T\subset\bbR^2$ to~$\Gamma_i$, let $s_1=|e_{ij_1}|$, and let~$s_2=|e_{ij_2}|$. By Lemma~\ref{lem:new_param}, we may assume that
\begin{equation*}
    \tilde T = \{ (\theta,\phi)\in\bbR^2\quad |\quad 0\leq \theta, \phi \;\text{ and }\; \theta/s_1 \leq 1- \phi/s_2  \},
\end{equation*}
and that $\tilde\bx$ maps the lines~$\phi=0$ and~$\theta=0$ onto~$e_{ij_1}$
and~$e_{ij_2}$, respectively. We now construct a reparameteriation $\Phi$
of~$\tilde T$, such that $\bx_i'=\tilde\bx\circ \Phi$ maps the edge~$\phi=0$
onto~$e_{ij_1}$ in an arclength preserving manner. Afterwards, we construct a
new reparameterization such that the edge $\theta=0$ is also mapped
onto~$e_{ij_2}$ in an arclength preserving manner.

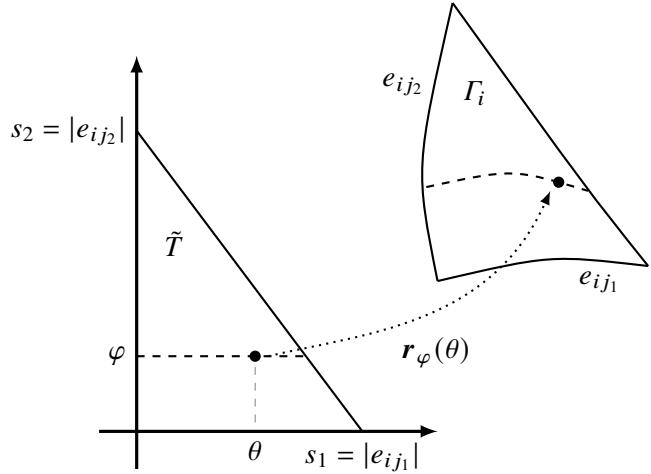
\begin{figure}
    \centering
    \begin{tikzpicture}
    \coordinate (O) at (0,0);
    \coordinate (A) at (3,0);
    \coordinate (B) at (0,4);
    \draw[black,thick] (O) -- (A) node [pos=0.515,below] {} node [below] {$s_1=|e_{ij_1}|$};
    \draw[black, thick] (O) -- (B) node [left] {$s_2=|e_{ij_2}|$};
    \draw[black, thick] (A) -- (B);
    \draw[black, thick, dashed] (0,1) -- (2.25,1) node[pos=0,left] {$\phi$} node [pos=0.7,circle,fill,inner sep=1.5pt] {} node [pos=0.7, pin={[pin distance=0.85cm]270:$\theta$}] {}; 
    
    \draw[black,very thick, -latex] (-0.5,0) -- (4,0) ;
    \draw[black,very thick, -latex] (0,-0.5) -- (0,5) ;
    \node(label) at (0.5,2.5) {$\tilde T$};

\coordinate (A') at (6.8,2.2);
\coordinate (B') at (4.2,5.7);

\coordinate (O') at (4,2);

\node(label) at (4.5,4.5) {$\Gamma_i$};

\draw[black, thick] (O') .. controls (5.5, 2.35) .. (A') node [pos=0.8, below] {$e_{ij_1}$};
\draw[black, thick] (O') .. controls (3.7, 3.5) .. (B') node [pos=0.8, left] {$e_{ij_2}$};

\draw[black, thick] (A') .. controls (6, 3.2) .. (B');
\draw[black, thick, dashed] (3.85, 3.25) ..  controls (5,3.5) .. (6, 3.2) node [pos=0.85,circle,fill,inner sep=1.5pt] {};

\draw[black, thick, dotted, -latex] (1.75, 1.0) ..  controls (3,1.3) and (4.7,1.4) .. (5.5, 3.2) node [pos=0.4,below right] {$\boldsymbol r_\phi(\theta)$};
    
    \end{tikzpicture}
    \caption[Reparamaterizing a face by arc-length]{The triangle $\tilde T$ is mapped to the face $\Gamma_i$. The $\theta$ and $\phi$ axes are mapped onto the edges $e_{ij_1}$ and $e_{ij_2}$, respectively. 
    The parameterization $\tilde{\bx}$ maps the line of constant $\phi$ to the curve $\bs{r}_\phi$ in $\Gamma_i$.
    In Lemma~\ref{lem:arc-length}, we construct a reparameterization $\Phi:\tilde T\to\tilde T$ such that this mapping is arc-length preserving for each $\phi$.}
    \label{fig:arc-reparam}
\end{figure}

We begin by defining the family of
curves~$\bs r_\phi(\theta) = \tilde \bx(\theta,\phi)$, see
Figure~\ref{fig:arc-reparam}. We also let the arclength along the
curve~$\bs r_\phi$ from~$\bs r_\phi(0)$ to~$\bs r_\phi(\theta)$ be denoted
by
\[
  s_\phi(\theta)= \int_0^\theta |\partial_\theta \tilde\bx(\theta',\phi)| du'.
\]
We define our reparameterization~$\Phi:\tilde T\to \bbR^2$ by
\begin{equation}
    \Phi(\theta,\phi) = \begin{pmatrix}\frac{s_\phi^{-1}(\theta)}{s_\phi^{-1}\lp s_1-\frac{s_1\phi}{s_2} \rp}\lp s_1-\frac{s_1\phi}{s_2}  \rp  \\\phi\end{pmatrix},
\end{equation}
where $s_\phi^{-1}$ is the inverse of $s_\phi$. We see that $\bx_i$ has the desired property that $\bx_i'(\theta,0)=\tilde \bx(\Phi(\theta,0))$ is an arclength parameterization of~$e_{ij_1}$ because
\begin{equation*}
    \left|\partial_\theta\bx_i'(\theta,0)\right|=\left|\partial_\theta[\tilde \bx(\Phi(\theta,0))]\right| = \left|\partial_\theta\tilde\bx|_{\Phi(\theta,0)}\;\frac{s_1\;|\partial_\theta\tilde\bx|^{-1}|_{\Phi(\theta,0)}}{s_0^{-1}(s_1)} \right| = 1.
\end{equation*}
To see that $\bx_i'$ is a smooth diffeomorphism mapping $\tilde T$ to $\Gamma_i$, it is enough to show that $\Phi(\tilde T)=\tilde T$ and that $\Phi$ is a smooth diffeomorphism from $\tilde T\to \tilde T$. The fact that $\Phi(\tilde T)=\tilde T$ follows from the monotonicity of $s_\phi^{-1}(\theta)$ with respect to $\theta$.
 To see that~$\Phi$ is a smooth map from~$\tilde T\to \tilde T$, we note that the inverse function theorem gives that
\begin{equation*}
    \frac{\lp1-\frac \phi{s_2} \rp}{s_\phi^{-1}\lp s_1-\frac{s_1\phi}{s_2} \rp}  = \frac{\lp1-\frac \phi{s_2} \rp}{\lp s_1-\frac{s_1\phi}{s_2} \rp|\partial_\theta \tilde\bx(\theta,\phi)|^{-1}+O\lp\lp1-\frac \phi{s_2} \rp^2\rp}.
\end{equation*}
Since the right hand side is a smooth function of~$\phi$, even near~$\phi=s_2$, the reparameterization $\Phi$ is smooth. A similar argument can be used to prove that~$\Phi^{-1}$ is also smooth so that $\Phi$ is a smooth diffeomorphism.

 We thus have that~$\bx_i'=\tilde\bx\circ \Phi$ is a smooth parameterization of~$\Gamma_i$ such that the edge~$\phi=0$ of~$\tilde T$ is mapped to~$\Gamma_i$ in an arc-length preserving manner. Repeating the above argument with $\theta$ and $\phi$ swapped will give a new parameterization $\bx_i$ such that both the edges~$\theta=0$ and~$\phi=0$ are mapped in an arc-length preserving manner.
\end{proof}
We note that the affine transformation~$\chi$ in the proof of Lemma~\ref{lem:new_param} preserves the arclength of any edge of~$T_i$ that gets mapped to an edge of~$\tilde T$ with the same length. Thus, if we compose the parameterization in the above proof with~$\chi$, we have a parameterization of~$\Gamma_i$ satisfying the requirements of the following theorem:
\begin{theorem}\label{thm:face_param}
  Let~$\Gamma_i$ be a face of a piecewise smooth surface~$\Gamma$ with edges~$e_{ij_1}$, $e_{ij_2}$, and~$e_{ij_3}$. Also let~$T\subset\bbR^2$ be any triangle such that two of its sides, $e_1$ and~$e_2$, have the same lengths~$e_{ij_1}$ and~$e_{ij_2}$. There exists a parameterization~$(T,\bx)$ of~$\Gamma_i$ such that $e_1$ and~$e_2$ are mapped onto $e_{ij_1}$ and~$e_{ij_2}$ in an arclength-preserving manner.
\end{theorem}
We now combine the parameterizations of all the faces that share a vertex~$S$ into a parameterization of the union of those faces,~$\Gamma_S$. This will give us the required piecewise smooth parameterization $\bx_S$.

\begin{proof}[Proof of Theorem~\ref{thm:acceptable_param}]
For every vertex~$S$ of~$\Gamma$, we define~$\eta_S=\{i_1,\ldots, i_{N_S}\}$ to be the set of faces that touch~$S$. We shall assume that~$\eta_S$ is sorted so that~$e_{i_j i_{j-1}}$ is non-empty for each~$j$, with~$j-1$ being interpreted mod~$N_S$. For each of these faces~$\Gamma_{i_j}$, we define~$T_{i_j}$ to be the triangle with vertices
\begin{equation*}
    \bs v_1 =\begin{pmatrix}0\\0\end{pmatrix},\quad{\bs v}_2= |e_{i_{j}i_{j-1}}|\begin{pmatrix}\cos \lp\frac{2\pi}{N_S} (j-1)\rp\\ \sin \lp\frac{2\pi}{N_S} (j-1)\rp\end{pmatrix},\quad \text{and}\quad {\bs v}_3=|e_{i_{j}i_{j+1}}|\begin{pmatrix}\cos\lp \frac{2\pi}{N_S} j\rp\\ \sin \lp\frac{2\pi}{N_S} j\rp\end{pmatrix}.
\end{equation*}
We use Theorem~\ref{thm:face_param} to pick a parameterization~$(T_{i_j},\bs x_{i_j})$ of~$\Gamma_{i_j}$ such that the edges of~$T_{i_j}$ touching the origin are mapped in an arc-length preserving manner to~$e_{i_j i_{j-1}}$ and~$e_{i_j i_{j+1}}$. The setup resulting from defining~$T_{i_j}$ and~$\bx_{i_j}$ for~$j=1,\ldots, N_S$ is depicted in Figure~\ref{fig:surf_param}.

\begin{figure}[t]
    \centering
    \begin{tikzpicture}
\coordinate (A) at (2.5,0);
\coordinate (B) at (0.62,1.9);
\coordinate (C) at (-0.97,0.71);
\coordinate (D) at (-2.43,-1.76);
\coordinate (E) at (0.53,-1.62);
\coordinate (O) at (0,0);
    
\draw[black, thick] (O) -- (A);
\draw[black, thick] (O) -- (B);
\draw[black, thick] (O) -- (E);
\draw[black, thick] (O) -- (C);
\draw[black, thick] (O) -- (D);

\draw[black, thick] (B)-- (A);
\draw[black, thick] (C) -- (B);
\draw[black, thick] (D) -- (E);
\draw[black, thick] (D) -- (C);
\draw[black, thick] (E) -- (A);
\filldraw[fill=gray!30, draw=black] (O) -- (A) -- (B) -- (O);

\draw[black,very thick,-latex] (0,-2) -- (0,2);
\draw[black,very thick,-latex] (-3,0) -- (3,0);

\node(label) at (0.8,0.8) {$T_{i_1}$};

\coordinate (A') at (7,3.2);
\coordinate (B') at (5,5.2);
\coordinate (C') at (4.3,3.7);
\coordinate (D') at (2.7,2.8);
\coordinate (E') at (6.2,2.1);

\coordinate (O') at (5,3);

\draw[black, thick] (O') .. controls (6, 3.5) .. (A');
\draw[black, thick] (O') .. controls (5.2, 3.5) .. (B');
\draw[black, thick] (O') .. controls (4.5, 3.3) .. (C');
\draw[black, thick] (O') .. controls (3.8, 2.5) .. (D');
\draw[black, thick] (O') .. controls (5.5, 2.9) .. (E');

\draw[black, thick] (A') .. controls (6, 4) .. (B');
\draw[black, thick] (B') .. controls (4.5, 4.2) .. (C');
\draw[black, thick] (C') .. controls (3, 3.55) .. (D');
\draw[black, thick] (D') .. controls (4, 2) .. (E');
\draw[black, thick] (E') .. controls (6.5, 2.2) .. (A');
\filldraw[fill=gray!30, draw=black] (O') .. controls (6, 3.5) .. (A') .. controls (6, 4) .. (B') .. controls (5.2, 3.5) .. (O');

\filldraw[black] (O') circle (2pt) node[anchor=north]{$S$};
\filldraw[black] (O) circle (2pt) ;

\node(label1) at (5.5,4) {$\Gamma_{i_1}$};
\node(label2) at (7.2,2.6) {$\Gamma_S$};

\draw[-latex,black,very thick, dashed] (1.7,0.3) .. controls (4,1) and (5,2) .. (6.2,3.7) node [pos=0.45, below right] {$\boldsymbol{x}_{i_j}$};

\end{tikzpicture}

    \caption[Parameterization of~$\Gamma_S$]{We construct a piecewise smooth parameterization of the surface~$\Gamma_S$ by piecing together parameterizations of each face of~$\Gamma_S$. Using Theorem~\ref{thm:face_param}, we pick parameterizations of each face~$\Gamma_{i_j}$ such that their domains to form a simple polygon.}
    \label{fig:surf_param}
\end{figure}
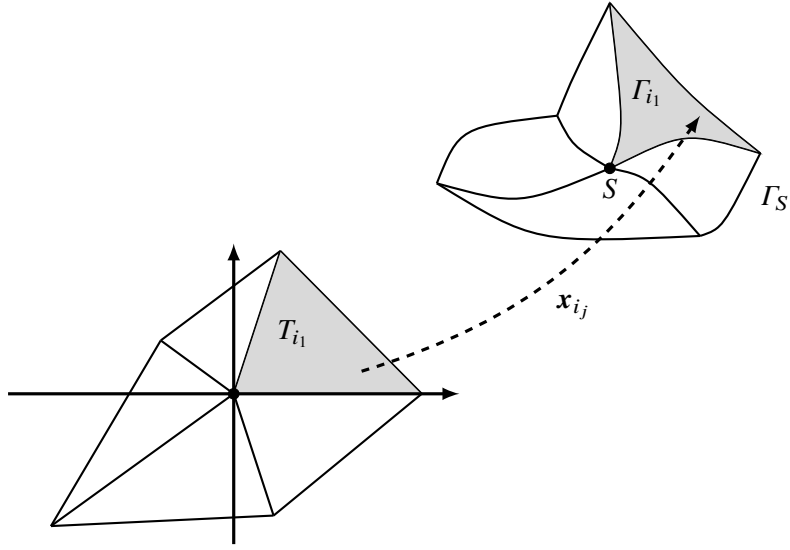

We then define the neighborhood~$U_S$ to be the interior of~$\cup_{j=1}^{N_S} T_{i_j}$ and define the map~$\bs x_s:U_S\to \Gamma$ piecewise by~$\bs x_S|_{T_{i_j}}=\bs x_{i_j}$. We note that because the~$\bx_{i_j}$'s were chosen to preserve arc-length on the edges contained in~$U_S$, the parameterizations $\bx_{i_j}$ and~$\bx_{i_{j-1}}$ will agree on the edge where their domains intersect. The map~$\bx_S$ will also be Lipschitz because each~$\bx_{i_j}$ is smooth. We also know that since~$\bs x_S|_{T_{i_j}}$ is smooth for each~$j$, the pull-back of the Laplace-Beltrami operator will be uniformly elliptic on the pre-image of each face, by Lemma~\ref{thm:unif_elip}. The smoothness of $\bx_S$ will also ensure that the functions $g^{-1}_i,h_i$ and $\alpha_i$ are smooth. The elliptic interface problem identified in Theorem~\ref{thm:equivalent_inter} is thus regular when $\bx_S$ is as defined above.
\end{proof}

We have now proved that there exists a picewise smooth parameterization such that the pull-back
of the Laplace-Beltrami problem from a single patch is a regular elliptic
interface form. In the next section, we identify the expansion pairs for each
vertex in that interface problem. This will allows us to apply
Theorems~\ref{thm:Nicaise} and~\ref{thm:Nicaise2} to the interface form of the
Laplace-Beltrami problem on a single patch.

\subsection{Solution of the interface form around a single vertex}\label{sec:interface_form_hols}

In the previous section, we identified the equivalence between the interface form of the Laplace-Beltrami problem
and an interface problem in the plane. We shall now leverage results in the plane to prove that the
weak solution of the Laplace-Beltrami problem solves the interface form. We
do this as follows: First, we state the weak form of the Laplace-Beltrami
problem on a single patch. Next we verify that the expansion powers for each
vertex are real so that Theorems~\ref{thm:Nicaise} and~\ref{thm:Nicaise2} can be
used to prove the equivalence of the weak and interface forms on a single patch
in Theorem~\ref{thm:single_patch_result}.

We shall consider the Laplace-Beltrami problem on a single patch with homogeneous Dirichlet boundary conditions. We first define the functions space we work in, then give the problem statement.
\begin{definition}
Let $\tilde \Gamma$ be a Lipschitz surface with boundary. The Sobolev space $H^1_0(\tilde\Gamma)$ is the set of $v$ in $H^1(\tilde\Gamma)$ with $\operatorname{tr}_{\partial\tilde\Gamma}v=0$.
\end{definition}

\begin{problem}[Weak form on a single patch]
Let $\tilde \Gamma$ be a Lipschitz surface with boundary and let $f\in L^2(\Gamma)$. A function $u\in H^{1}_0(\tilde\Gamma)$ solves the the weak form of the Laplace-Beltrami problem on $\tilde \Gamma$ with homogeneous Dirichlet boundary conditions on $\partial\tilde\Gamma$ if
\begin{equation}
    -\int_{\tilde\Gamma} \gradg u \cdot \gradg v = \int_{\tilde \Gamma} fv, \qquad \text{for all } v\in H^{1}_0(\tilde\Gamma). \label{eq:Weaksinglepatch}
\end{equation}
\end{problem}

In order to apply Theorem~\ref{thm:Nicaise} in a useful manner, we
 need the expansion pairs for each vertex of the surface $\Gamma$. As these only depend on the structure of the elliptic interface problem in an infinitesimal region of the vertex, and the surfaces we consider are piecewise smooth, it is enough to compute the expansion pairs for a corner formed by flat faces and straight edges. An example of such a geometry and the coordinate
system that we consider is in~Figure~\ref{fig:ConeDiagram}. We summarize the
result in the following lemma.

 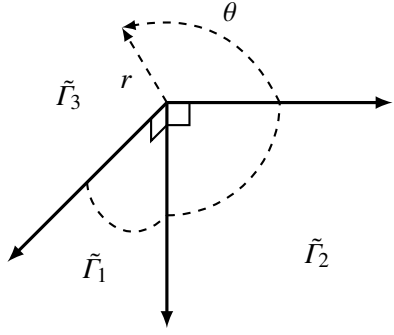
\begin{figure}[t]
    \centering
\begin{tikzpicture}
\coordinate (O) at (0,0);
\coordinate (A) at (0,-3);
\coordinate (B) at (3,0);
\coordinate (C) at (-2.12,-2.12);

\coordinate (A') at (0,-.3);
\coordinate (B') at (.3,0);
\coordinate (C') at (-0.21,-0.21);

\draw[black,very thick,-latex] (O) -- (A);
\draw[black,very thick,-latex] (O) -- (B);
\draw[black,very thick,-latex] (O) -- (C);

\draw[black, thick] (A') -| (B');
\draw[black, thick] (C') -- (-0.21,-0.51) -- (A');


\node(label) at (2,-2) {$\tilde\Gamma_2$};
\node(label) at (-0.95,-2.2) {$\tilde\Gamma_1$};
\node(label) at (-1.3,0.1) {$\tilde\Gamma_3$};

\coordinate (R) at (-0.6,1);
\draw[thick, black, dashed] (-1.06,-1.06) .. controls (-1.06, -1.4155) and (-0.555,-2.0555) .. (0,-1.5); 
\draw[thick, black, dashed] (0,-1.5) arc (270:360:1.5cm);
\draw[thick, black, dashed, -latex] (1.5,0) .. controls (1.,1) and (0.3,1.155) .. (R) node [pos=0.5, above right] {$\theta$};
\draw[thick, black, dashed, -latex] (O) -- (R) node [pos=0.5, below left] {$r$};

\end{tikzpicture}
    \caption[A surface composed of planar wedges]{An example of a surface $\tilde\Gamma$ composed of three
      infinite planar wedges. In this example, the wedge angles
      are~$\gamma_1=\gamma_2={\pi}/2$ and~$\gamma_3={3\pi}/2$. The
      conic angle is~${5\pi}/2$.}
    \label{fig:ConeDiagram}
\end{figure}
\begin{lemma}\label{lem:corners}
Suppose that~$\tilde \Gamma_1,\ldots,$ $\tilde \Gamma_n$ are a
collection of two-dimensional infinite planar wedges embedded
in~$\bbR^3$.
Let~$\tilde \Gamma$ be the union
of these faces and suppose that all
faces meet at a single vertex, and that~$\tilde \Gamma_i$ shares an edge with~$\tilde \Gamma_{i+1}$ for $i=1,\ldots,n-1$. Lastly, let each face be parameterized by
the polar coordinates~$(r_i,\theta_i)$ with~$\theta_i=0$ on the edge~$e_{i(i-1)}$ and~$\theta_i=\gamma_i$ on~$e_{i(i+1)}$. The piecewise
parameterization of~$\tilde \Gamma$ is then
\begin{equation}
    (r,\theta)=
    (r_i,\theta_i+\sum_{j=1}^{i-1}\gamma_j)  \text{ on } \tilde\Gamma_i.
\end{equation}
The sum $\gamma:=\sum_{i=1}^n\gamma_i$ is known as the conic
angle of the vertex.  If $B_R$ is a ball of radius $R$ centered on the
vertex and $u\in H^{1}(B_R\cap \tilde \Gamma)$ is a solution of
$\Delta_{\tilde\Gamma} u=0$, then one of the following holds.

\begin{enumerate}
    \item The wedges $\tilde \Gamma_1$ and $\tilde \Gamma_n$ share an edge and for all $r < R$ the solution $u$ can
be written as \begin{equation}
  \label{eq:ucorner}
    u(r,\theta)= \sum_{n=-\infty}^\infty a_n \, 
    r^{\frac{2\pi}{\gamma}|n|} \, e^{i\frac{2\pi}{\gamma}n\theta}
\end{equation}
for some set of coefficients~$a_n \in \mathbb C$. Furthermore, if~$k
<1+{2\pi}/\gamma$, then on each face~$\tilde{\Gamma_i}$ we have
that~$u\in H^{k}(B_R\cap \tilde \Gamma_i^{\mathrm{o}})$.
\item The wedges $\tilde \Gamma_1$ and $\tilde \Gamma_n$ do not share an edge, we enforce that~$\tr_{\partial\tilde\Gamma}u=0$, and for all $r < R$ the solution $u$ can
be written as
\begin{equation}
  \label{eq:ucorner_hom}
    u(r,\theta)= \sum_{n=1}^\infty a_n \, 
    r^{\frac{\pi}{\gamma}|n|} \, \sin\lp \frac{\pi}{\gamma}n\theta\rp
\end{equation}
for some set of coefficients~$a_n \in \mathbb C$. Furthermore, if~$k
<1+{\pi}/\gamma$, then on each face~$\tilde \Gamma_i$ we have
that~$u\in H^{k}(B_R\cap \tilde \Gamma_i^{\mathrm{o}})$.
\end{enumerate}
\end{lemma}
Before we prove this lemma, we make a few observations.
First, we note that in the case where~$\tilde \Gamma_1,\ldots, \tilde
\Gamma_n$
are co-planar and $\tilde \Gamma_1$ and $\tilde \Gamma_n$ share an edge,~$\tilde \Gamma$ will be a plane in~$\bbR^3$. The conic
angle~$\gamma$ will be~$2\pi$ and
so~\eqref{eq:ucorner} becomes the usual separation of variables solutions
to Laplace's equation in the plane and is therefore smooth, as we would expect. We also note that if a vertex~$S$ of~$\Gamma$ is artificial, in the sense that it was added in order to triangulate the surface, the conic angle at that vertex will be~$2\pi$. The solution will thus be smooth there, as we would expect. We now prove the lemma.
\begin{proof}
 We begin with the case where $\tilde \Gamma_1$ and $\tilde \Gamma_n$ share an edge, so that vertex is an interior one. We shall use separation of variables and look for a solution of~$\Delta_{\tilde\Gamma}u_n=0$ of the form~$u_n=c_n(r)\tau_n(\theta)$. In polar coordinates the Laplace-Beltrami operator will be
 \begin{equation}
     \Delta_{\tilde\Gamma}  = \frac{\partial^2}{\partial r^2}+\frac1r\frac{\partial}{\partial r}+\frac1{r^2}\frac{\partial^2}{\partial \theta^2}.
 \end{equation}
 Plugging~$u_n$ into this formula will give that~$\tau_n(\theta) = e^{i\lambda_n\theta}$. Using the fact that~$\bx(r,0)=\bx(r,\gamma)$ will give that~$\lambda_n=\frac{2\pi}{\gamma} n$ for~$n\in\mathbb{N}$. Some more calculus will give that
 \begin{equation}
  \begin{aligned}
    c_0(r)&=a_0+b_0\log(r), \\
    c_n(r)&=a_n \,  r^{\frac{2\pi}{\gamma}|n|}+b_n \,
    r^{-\frac{2\pi}{\gamma}|n|}, \qquad \text{for } n \neq 0.
  \end{aligned}
\end{equation}
If we now impose that~$u_n\in  H^1(B_R\cap \tilde
 \Gamma)$, then we find that
 \begin{equation}
     u_n = a_n \,  r^{\frac{2\pi}{\gamma}|n|}\,e^{i\frac{2\pi}{\gamma}n\theta}.
 \end{equation}
It remains to show that we can write the solution~$u$ as sum of the~$u_n$'s. 

We pick the Fourier coefficients~$a_n$ so that $\sum_n u_n$ agrees with~$u$ on~$\partial B_R\cap\tilde \Gamma$. The uniqueness of weak solutions to the Laplace-Beltrami problem on~$B_R\cap\tilde \Gamma$ with Dirichlet boundary conditions (Lemma~\ref{thm:uniqueness} below) will then give that~$u-\sum_n u_n\equiv 0$ on~$B_R\cap \tilde
 \Gamma$. We have thus proved~\eqref{eq:ucorner}.
 
 In order to see the higher regularity of $u$ on each face, we note
 that the most singular terms in the expansion occur when~$n = \pm 1$.
 These terms will be contained in~$H^{k}(B_R\cap \tilde \Gamma_i^{\mathrm{o}})$ if and only if
 \begin{equation*}
     \left\Vert \nabla^k r^{\frac{2\pi}\gamma}e^{\pm
       i\frac{2\pi}\gamma}\right\Vert_{L^2(B_R\cap \tilde \Gamma_i^{\mathrm{o}})}\propto
     \left(\int_0^R
     r^{2(\frac{2\pi}\gamma-k)+1}dr\right)^{1/2}<\infty.
 \end{equation*}
 This occurs precisely when~$k<1+\frac{2\pi}\gamma$.
 
 The proof when $\tilde \Gamma_1$ and $\tilde \Gamma_n$ do not share an edge is
 identical to the case above, except that the Fourier series are replaced by
 sine series because of the boundary conditions.
\end{proof}

\begin{remark}
  We will use the above lemma to determine the singularities in the solution of
  the Laplace-Beltrami problem on surfaces with corners. The singularities
  in~\eqref{eq:ucorner} were first derived in~\cite{Buffa2002} for the special
  case of polyhedral surfaces, i.e. surfaces with only flat faces and straight
  edges, but we have presented Lemma~\ref{lem:corners} to make this section as
  complete as possible. In what follows, we will prove that the solution has
  these same singularities in the more general case where the surface faces are
  curved.
\end{remark}

Now that we know the nature of the singularities of the solution of the
interface form near interior and exterior corners of a surface, we may prove the
following theorem.

\begin{lemma}\label{thm:interfaceholds}
  Let~$\Gamma_S$ be the union of the faces of~$\Gamma$ that touch the
  vertex~$S$. Let~$\Gamma_S$ be parameterized by the piecewise smooth map~$\bx_S$ with
  domain~$U_S$ (from Theorem~\ref{thm:acceptable_param}), and let~$f$ be a
  function in~$\mathcal H^k(\Gamma_S)$ with~$k=0,1,$ or $2$. Also let~$L$ be the
  pull-back of the Laplace-Beltrami operator by~$\bx_S$. If~$\tilde u$ is the
  solution of the elliptic interface problem with homogeneous Dirchlet boundary
  conditions~\eqref{eq:elliptic_in_plane} on~$U_S$ with right hand
  side~$f\circ \bx_S\in\mathcal H^k(U_S)$, then~$u=\tilde u\circ \bx_S^{-1}$ is
  the unique solution of the interface forms of the Laplace-Beltrami
  problem~$\Delta_{\Gamma_S}u=f$ with homogeneous Dirichlet boundary conditions
  on~$\partial\Gamma_S$.

  Further, there exists a neighborhood of each vertex~$S'$ of~$\Gamma_S$ we may
  write~$u$ as
  \begin{equation}
    \label{eq:local_corner}
    u=u_0+
    \begin{cases}
              \sum_{|n|\leq
          \frac{\gamma_S}{2\pi}(k+3) } a_{n,S'} \, 
    r_{S'}^{\frac{2\pi}{\tilde\gamma_{S'}}|n|} \, e^{i\frac{2\pi}{\tilde\gamma_{S'}}n\theta_{S'}}, &\text{if } S'=S,\\[8pt]
    \sum_{|n|\leq
          \frac{\gamma_S}{\pi}(k+3)} a_{n,S'} \, 
    r_{S'}^{\frac{\pi}{\tilde\gamma_{S'}}|n|} \, \sin\lp\frac{\pi}{\tilde\gamma_{S'}}n\theta_{S'}\rp, &\text{otherwise},
  \end{cases}
  \end{equation}
  where~$u_0\in\mathcal{H}^{2+k}(\Gamma_S)$ and~$\gamma_{S'}$ is the conic angle of the vertex~$S'$, i.e. the sum of the corner angles of all of the faces in~$\Gamma_S$ touching the vertex~$S'$.
\end{lemma}
\begin{proof}
  Let~$\mathcal{C}_S$ be the set of corners of~$\Gamma$ that are included
  in~$\Gamma_S$ and let~$\tilde \gamma_{S'}$ be the conic angle of the
  vertex~$S'\in\mathcal{C}_S$ with respect to the surface~$\Gamma_S$. Since the
  radial distance in the plane and on the surface will be proportional in the
  limit as they get small, Lemma~\ref{lem:corners} tells us that the expansion
  powers at the corner~$S'$
  are~$\{{2\pi n}/{\tilde\gamma_{S'}}\}_{n\in\mathbb{Z}}$ if $S'=S$
  or~$\{{\pi n}/{\tilde\gamma_{S'}}\}_{n\in\mathbb{Z}}$ if $S'\neq S$. As
  these are all real and we have already verified that the planar elliptic
  interface problem is regular, we have thus verified that we may apply
  Theorem~\ref{thm:Nicaise}. The solution~$\tilde u$
  of~\eqref{eq:elliptic_in_plane} thus exists and is unique.
  Theorem~\ref{thm:equivalent_inter} then gives that~$\tilde u\circ \bx_S^{-1}$
  is the unique solution of the interface form of the Laplace-Beltrami problem.

Theorem~\ref{thm:Nicaise2} will give the singularity expansions listed in~\eqref{eq:local_corner}.
\end{proof}

We now prove that the solution of the interface form of the Laplace-Beltrami problem is the unique solution of the weak form.

\begin{lemma}\label{thm:uniqueness}
  The weak solution of the Laplace-Beltrami problem on $\Gamma_S$ with homogeneous Dirichlet boundary conditions on $\partial\Gamma_S$ is unique.
\end{lemma}
\begin{proof}
By linearity, we need only check the case when $f=0$. We recall that $u\in H^1_0(\Gamma_S)$ is the weak solution if
\begin{equation}
    \int_{\Gamma_S} \gradg u \cdot \gradg v = 0 \quad \forall v\in H^1_0(\Gamma_S).
\end{equation}
Substituting $u=v$ into the expression gives that $\int_{\Gamma_S} |\gradg u|^2=0$, so $u$ must a constant. The boundary conditions then imply that $u\equiv 0$, and so the solution is unique.
\end{proof}

\begin{theorem}\label{thm:single_patch_result}
  The solution of the interface form of the Laplace-Beltrami problem on $\Gamma_S$ with homogeneous Dirichlet boundary conditions on $\partial\Gamma_S$ exists and is the unique weak solution.
\end{theorem}
\begin{proof}
It only remains to show that the solution of the interface form identified in Lemma~\ref{thm:interfaceholds}~$u$ is a weak solution of the Laplace-Beltrami operator. To prove this, we let $v$ be a test function in $H^1_0(\Gamma_S)$ and let~$\Gamma_{i,\epsilon}=\Gamma_i \backslash \cup_{S'\in\mathcal{C}_S} B_\epsilon(S')$. Since the expansion~\eqref{eq:local_corner} implies that~$u\in
\mathcal{H}^2(\cup_{i\in\eta_S}\Gamma_{i,\epsilon})$, we have that
\begin{multline}
    -\int_{\cup_{i\in\eta_S}\Gamma_{i,\epsilon}} \nabla_{\Gamma}v\cdot\nabla_{\Gamma} u =
    \sum_i\int_{\Gamma_{i,\epsilon}}
    v\nabla_{\Gamma}\cdot\nabla_{\Gamma} u +\int_{e_{i,0}\backslash \cup_{S'}}v\frac{\partial u}{\partial\bs b_i}\\
    +    \sum_{ij}\int_{e_{ij}\backslash \cup_{S'}
      B_\epsilon(S')}v\left(\frac{\partial u}{\partial \binormal}
    +\frac{\partial u}{\partial \binormalj}\right) -\sum_{S'}
    \int_{\partial B_\epsilon(S')} v \hat{\bs r}_{S'} \cdot \nabla_\Gamma
    u.
\end{multline}
Since~$u$ solves the interface problem on~$\Gamma_S$, and~$v\in H^1_0(\Gamma_S)$, most of the boundary terms vanish and we are left with
\begin{equation}
    -\int_{\cup_{i\in\eta_S}\Gamma_{i,\epsilon}} \nabla_{\Gamma}v\cdot\nabla_{\Gamma} u =
    \int_{\cup_{i\in\eta_S}\Gamma_{i,\epsilon}} vf +\sum_{S'} \int_{\partial B_\epsilon(S')} v
    \hat{\bs r}_{S'} \cdot \nabla_\Gamma u.
  \end{equation}

  We now show that the final term vanishes as~$\epsilon \to 0$. First,
  suppose that~$v$ is bounded by a constant~$C$ and that~$S'=S$. If~$\epsilon$
  is small, then~\eqref{eq:local_corner} holds. The contribution of~$u_0$ to
  this integral will vanish as~$\epsilon \to 0$
  since~$u\in \mathcal{H}^2(\Gamma_S)$. To see that that the singular terms
  vanish, we replace the metric in the surface gradient with its limit in the
  corner, denoted by~$\nabla_{\Gamma_S}$, introducing an error that shrinks
  with~$\epsilon$:
\begin{equation}
    \left|\int_{\partial B_\epsilon(S)} v \hat{\bs r}_S \cdot \nabla_\Gamma r_S^{\frac{2\pi}{\gamma_S}|n|} \,
        e^{i\frac{2\pi}{\gamma_S}n\theta_S}\right| \leq C \int_{\partial B_\epsilon(S)}   \left| \hat{\bs r}_S \cdot \nabla_{\Gamma_S} r_S^{\frac{2\pi}{\gamma_S}|n|} \,
        e^{i\frac{2\pi}{\gamma_S}n\theta_S}\right|
        + o\lp|\partial B_\epsilon(S)|\rp.
\end{equation}
Evaluating the remaining derivative and integrating
gives that
\begin{equation}
    \left|\int_{\partial B_\epsilon(S)} v \hat{\bs r}_S \cdot
    \nabla_\Gamma r_S^{\frac{2\pi}{\gamma_S}|n|} \,
    e^{i\frac{2\pi}{\gamma_S}n\theta_S}\right| \leq C \gamma_S \epsilon
    \lp \frac{2\pi}{\gamma_S}|n|
    \epsilon^{\frac{2\pi}{\gamma_S}|n|-1} \rp + o\lp\gamma_S\epsilon\rp.
\end{equation}
The right hand side above vanishes as~$\epsilon\to 0$, and so
\[
  \int_{\partial B_\epsilon(S)} v
  \hat{\bs r}_{S'} \cdot \nabla_\Gamma u\to 0.
\]
We may repeat the above argument for the other vertices of~$\Gamma_S$ to see
that for any bounded~$v\in H^1_0(\Gamma_s)$,
\begin{equation}
   -\int_{\Gamma_S} \nabla_{\Gamma}v\cdot\nabla_{\Gamma} u =
   \int_{\Gamma_S} vf .
\end{equation}
Since bounded functions are dense in~$H^1_0(\Gamma_s)$, we have shown that~$u$
solves the weak form of the Laplace-Beltrami problem on~$\Gamma_S$. Finally, Lemma~\ref{thm:uniqueness} thus gives that $u$ is the unique weak solution.
\end{proof}
\subsection{Interface form on a closed surface}\label{sec:closed_surface}

We now extend the result for a single surface patch~$\Gamma_S$ to rigorously connect the weak and interface forms of the
Laplace-Beltrami problem on the whole surface~$\Gamma$, which is the main result of this paper.

\begin{theorem}[Equivalence of the interface form]
  \label{thm:strongLB}
  If $\Gamma$ is a piecewise smooth Lipschitz surface without boundary and $f\in
  \mathcal{H}^k(\Gamma)$ with $k=0,1,$ or $2$ is mean-zero, then the following
  hold:
  \begin{enumerate}
    \item The weak solution of the Laplace-Beltrami problem~$\LB u =f$ solves the interface form of the Laplace-Beltrami problem~\eqref{eq:interface_form}.
    \item There exists a smooth partition of unity over~$\Gamma$, $\{\zeta_S\}_{S\in \mathcal{C}}$, such that the support of~$\zeta_S$ is strictly contained in~$\Gamma_S$ and there exists a~$u_0\in \mathcal{H}^{2+k}(\Gamma)$ such that
      \begin{equation} \label{eq:expandu}
        u=u_0+ \sum_{S\in \mathcal{C}} \zeta_S \left( \sum_{|n|\leq
          \frac{\gamma_S}{2\pi}(k+3) }
        a_{n,S} \, r_S^{\frac{2\pi}{\gamma_S}|n|} \,
        e^{i\frac{2\pi}{\gamma_S}n\theta_S}\right), 
      \end{equation}
 where $\gamma_S$ is the conic angle for the corner $S$ and
 $(r_S,\theta_S)$ is a local polar coordinate system defined around
 the corner $S$, as previously described in Lemma~\ref{lem:corners}.
\end{enumerate}

\end{theorem}
\begin{proof}
Let~$\{(U_S,\bx_S)\}_{S\in \mathcal{C}}$ be the collection of the local parameterizations defined in the proof of Theorem~\ref{thm:acceptable_param}. To construct the partition of unity, we define $\tilde \zeta_{\bx_S^{-1}(S)}$ as in the proof of Lemma~\ref{lem:part_of_unity} for each $S\in\mathcal{C}$. We then define $\tilde \zeta_S\in C(\Gamma)$ to be $\tilde \zeta_{\bx_S^{-1}(S)}\circ \bx_S^{-1}$ on $\Gamma_S$ and zero on $\Gamma\setminus \Gamma_S$. The functions $\zeta_S = \tilde\zeta_S/\sum_{S'\in\mathcal{C}}\tilde\zeta_{S'}$ are then smooth on each face, satisfy the interface conditions, and form the desired partition of unity.

Now that we have a suitable partition of unity, we define
\begin{equation}
    \tilde f_S := \LB (\zeta_S u),\label{eq:ftilde}
\end{equation} 
for each $S\in\mathcal{C}$. In order to apply Theorem~\ref{thm:single_patch_result}, we must verify that~$\tilde f_S\in L^2(\Gamma)$. To do this, we use the product rule to see that
\begin{equation}
    \tilde f_S =\divg(u\gradg\zeta_S)+\divg(\zeta_S\gradg u)\label{eq:ftilde_exp}.
\end{equation}
 For any~$v\in\Ho$, the first term is defined by
\begin{equation}
    (\divg(u\gradg\zeta_S),v) := -\int_\Gamma u\gradg \zeta_S \cdot\gradg v = \sum_i \int_{\Gamma_i} \nabla_{\Gamma_i}\cdot(u\nabla_{\Gamma_i}  (\zeta_S )) v + \sum_j \int_{e_{ij}} \bs b_i \cdot \nabla_{\Gamma} (\zeta_S) uv,
\end{equation}
where we have used integration by parts on each face. The edge integrals cancel
because the trace of~$u$ and~$v$ along~$e_{ij}$ agrees from~$\Gamma_i$ and~$\Gamma_j$, since they are in~$\Ho$, and~$\zeta_S$ satisfies the interface conditions by construction. Applying the product rule on each face gives that
\begin{equation}
    (\divg(u\gradg\zeta_S),v) = \sum_i \int_{\Gamma_i} \lp\nabla_{\Gamma_i}u\cdot\nabla_{\Gamma_i}\zeta_S + u \Delta_{\Gamma_i}\zeta_S \rp v.
\end{equation}
Since~$\zeta_S$ is smooth on each~$\Gamma_i$ and~$u\in\Ho$, the function in brackets is in~$L^2(\Gamma)$. The density of $\Ho$ in $L^2(\Gamma)$ then gives that $\divg(u\gradg\zeta_S)\in L^2(\Gamma)$. 

For any~$v\in\Ho$, the second term in~\eqref{eq:ftilde_exp} is
\begin{equation}
     (\divg(\zeta_S\gradg u),v) := -\int_\Gamma \zeta_S\gradg u \cdot\gradg v = -\int_\Gamma \gradg u \cdot \lp \zeta_S\gradg v\rp. 
\end{equation}
Applying the product rule and using the fact that~$\zeta_S$ is Lipschitz continuous gives that
\begin{equation}
     (\divg(\zeta_S\gradg u),v)=-\int_\Gamma \gradg u \cdot \lp \gradg (\zeta_S v) -v\gradg \zeta_S\rp=(\divg(\gradg u),\zeta_Sv)+\int_\Gamma \gradg\zeta_S\cdot \gradg uv. 
\end{equation}
The density of $\Ho$ in $L^2(\Gamma)$ then gives that~$(\divg(\zeta_S\gradg u)\in L^2(\Gamma)$. Overall, we now have that~$\tilde f_S\in L^2(\Gamma)$.

We pause here to note that if $\zeta_S$ had not satisfied the interface conditions, then~$\tilde f_S$ would have involved a distribution supported on the surface edges and not have been in~$L^2(\Gamma)$.

Since~$\zeta_S u\in H^1_0(\Gamma_S)$, the function $\zeta_Su$ is a solution of the weak form of the Laplace-Beltrami problem on~$\Gamma_S$ with homogeneous Dirichlet boundary conditions and right hand side~$\tilde f_S$. Since~$\tilde f_S\in L^2(\Gamma)$, Theorem~\ref{thm:single_patch_result} thus gives that~$\zeta_S u$ satisfies the interface conditions on the interior of~$\Gamma_S$. Since~$\zeta_S$ is zero in a neighborhood of~$\Gamma\setminus\Gamma_S$, we know that~$\zeta_S u$ satisfies the interface conditions at all the other edges of~$\Gamma$. The linearity of our problem then tells us that~$u$ satisfies~\eqref{eq:interface_form}.

For the second part of the theorem, we note that~$\zeta_S$ is flat near $S$ and
zero in a neighborhood every other vertex. The expansion of~$\zeta_Su$ given by
\eqref{eq:local_corner} therefore gives the expansion of~$u$ near~$S$. The
global form of~$u$ in~\eqref{eq:expandu} can then be found using the argument in
Corollary~\ref{cor:cornerstodomain} with~$\{\zeta_S\}$ as the partition of
unity.

If $f\in H^1(\Gamma)$, then we continue to show that the coefficients $a_{n,s}$ can be chosen so that $u_0$ is in $\mathcal{H}^3(\Gamma)$. Since~$\zeta_S$ must be flat near the vertices of~$\Gamma$,~\eqref{eq:ftilde} implies that that~$\tilde f_S\in \Ho$. With this information, the above proof then implies that~$u_0$ can be chosen to be in $\mathcal{H}^3(\Gamma)$. If~$f\in\Htg$, we may repeat this argument to see that $u_0$ can be chosen to be in $\mathcal{H}^4(\Gamma)$. Having exhausted the relevant cases, we have completed the proof.
\end{proof}

Morrey's
inequality (see Theorem 12.55 in~\cite{Leoni2017}) indicates that if one could prove a version of the above
for~$f\in L^p(\Gamma)$ for some~$p>2$, then we could ask for~$u$ to be
the solution of the following smoother problem:
 \begin{problem}[Strong interface form of the Laplace-Beltrami problem]\label{def:strong_form}
Let~$\Gamma$ be a surface composed of smooth faces~$\Gamma_i$ and $f$ be a continuous mean-zero function on $\Gamma$.  The
strong interface form of the Laplace-Beltrami problem is
defined to be: find a~$u\in C^1(\Gamma_i)\cap H^2(\Gamma_i)$ for every~$i$, which satisfies~\eqref{eq:interface_form} where the restrictions to surface
edges are interpreted in the limit sense, rather than a trace
sense.
\end{problem}

We present the following conjecture about conditions for solvability
of the strong interface form. The proof would follow from an~$L^p$
version of Theorem~\ref{thm:Nicaise} in the same way presented above
and an application of Morrey's inquality.
\begin{conjecture}\label{cnj:lpresults}
If~$\Gamma$ is a piecewise smooth Lipschitz surface composed of faces~$\Gamma_i$ with all conic angles less than or equal to~$2\pi$
and~$f$ is a function in~$L^p(\Gamma)$ for some~$p>2$, then the solution of the weak form~\eqref{eq:weak_form} is also a solution of the strong interface form of the Laplace-Beltrami problem. 
\end{conjecture}

The results of this section imply that numerical solvers for the
Laplace-Beltrami problem on piecewise smooth surfaces can discretize
the differential operator acting on functions in~$L^2$ and obtain
convergent results. We demonstrate this in Section~\ref{sec:ode} via a
high-order numerical solver along surfaces of revolution. Before doing that however, we briefly discuss how the above result can be extended to cones.

\section{A special case: The cone}
\label{sec:cone}

Definition~\ref{def:piecewise_smth} excluded surfaces with cone-like
singularities. This exclusion allowed us to simplify our arguments, but many applications involve surfaces with such singularities. In this
section, we briefly address the Laplace-Beltrami problem on a simple cone of
height~$h$ with a base of radius one. 

We begin by analytically solving the Laplace-Beltrami
problem on this surface, showing that the expected singularities appear, and
noting that the solution has the expected smoothness.

\begin{example} \label{exam:cone}
Consider the case where $\Gamma$ is a cone of
height~$h$ with a base of radius one. We let~$\Gamma_1$ be the curved part 
of the cone and let $\Gamma_2$ be the flat bottom of the cone. These pieces can
be parameterized by
\begin{equation}
  \begin{aligned}
    \bs x_1 &=  \lp \frac{h-z}h\cos(\phi),
      \frac{h-z}h\sin(\phi), z \rp, &\qquad &z \in [0,h], &\quad &\phi \in
      [0,2\pi), \\
    \bs x_2 &=\lp r\cos(\phi), r\sin(\phi), 0\rp, & &r \in [0,1], &
    &\phi \in [0,2\pi),
  \end{aligned}
\end{equation}
respectively. Using these parameterizations~$\bs x_1$ and~$\bs x_2$,
Mathematica~\cite{Mathematica} can analytically solve the Laplace-Beltrami
problem with piecewise data~$f_n$ given by:
\begin{equation}
  \begin{aligned}
    f_n|_{\Gamma_1} &= (h-z)^\alpha \, e^{in\phi}, \qquad \alpha > -1,
    \quad n  \in \mathbb{N}, \\
    f_n|_{\Gamma_2} &= 0.
  \end{aligned}
\end{equation}
The analytic solution is given by
\begin{equation}
  \begin{aligned}
    u_n|_{\Gamma_1} &= \lp c_1 \, (h-z)^{2+\alpha} + c_2 \,
    (h-z)^{\sqrt{1+h^2} |n|}\rp
    e^{in\phi}, \\
    u_n|_{\Gamma_2} &= c_3\, r^{|n|} \, e^{in\phi},
  \end{aligned}
\end{equation}
where $c_1$, $c_2$, and $c_3$ are constants that depend on $h$, $n$, and
$\alpha$. We can clearly see that in this case the solution has two more
integrable derivatives than~$f_n$ and picks up terms whose higher order
derivatives are singular near the corner. If we note that the conic angle for
this cone is given by $\gamma={2\pi} / \sqrt{1+h^2}$, then it becomes even
clearer that the solution has the expected behavior near the top tip and the
bottom edge.
\end{example}

Having seen that the solution of the Laplace-Beltrami problem has the expected behaviour on~$\Gamma$, we
now prove that Theorem~\ref{thm:strongLB} can be applied to this
specific~$\Gamma$. Since the cone is piecewise smooth away from its natural
vertex~$S=(0,0,h)$, it is enough to show that Lemma~\ref{thm:interfaceholds}
holds on a patch containing the vertex $S$. The main requirement for this is
included in the following lemma.
\begin{lemma}
    There exists a parameterization~$\bx_S$ of a patch~$\Gamma_S$ around the vertex~$S$ such that the pull-back of the Laplace-Beltrami operator has piecewise constant coefficients. 
\end{lemma}
\begin{proof}
    We begin by parameterizing the curved section of the cone~$\Gamma_1$ as
\begin{equation}
    \tilde{\bs  x}_1(x,y) =\begin{pmatrix}\sqrt{\frac{x^2+y^2}{1+h^2}}\cos\left(\sqrt{1+h^2}\tan^{-1}\left(\frac yx\right) \right)\\
    \sqrt{\frac{x^2+y^2}{1+h^2}}\sin\left(\sqrt{1+h^2}\tan^{-1}\left(\frac yx\right) \right)\\
    h\sqrt{\frac{x^2+y^2}{1+h^2}}\end{pmatrix}, \qquad \text{for } (x,y)\in W
\end{equation}
where the wedge $W$ is
\begin{equation}
    W=\left\{(x,y)\in \mathbb{R}^2\;\middle|\;\sqrt{x^2+y^2}\leq \sqrt{1+h^2}\text{ and }  \tan^{-1}\left(\frac yx\right)\in \left[0,\frac{2\pi}{\sqrt{1+h^2}}\right)\right\}.
\end{equation}
Some calculus shows that under this parameterization, the pull-back of the
Laplace-Beltrami operator is simply the Laplacian. To construct the patch~$\Gamma_S$, we let~$T_1,T_2$, and~$T_3$ be the three triangular subsets of~$W$ shown in~Figure~\ref{fig:cone_param}. For the purposes of this proof, we consider their images~$\Gamma_i=\tilde{\bx}_1(T_i)$ to be faces of~$\Gamma$. The patch around the vertex is then defined as~$\Gamma_S:=\cup_{i=1}^3\Gamma_i$.

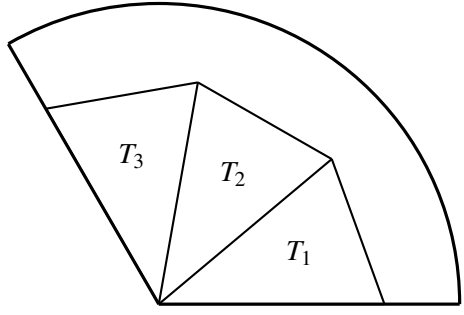
\begin{figure}
    \centering
    \begin{tikzpicture}
        \coordinate (0) at (0,0);
        \coordinate (A) at (3,0);
        \coordinate (B) at (2.30,1.93);
        \coordinate (C) at (0.52,2.95);
        \coordinate (D) at (-1.5,2.60);
        
        \coordinate (A') at (4,0);
        \coordinate (D') at (-2,3.46);

        \draw[black,thick] (0) -- (A);
        \draw[black,thick] (0) -- (B);
        \draw[black,thick] (0) -- (C);
        \draw[black,thick] (0) -- (D);
        \draw[black,thick] (A) -- (B);
        \draw[black,thick] (B) -- (C);
        \draw[black,thick] (C) -- (D);

        \draw[black,very thick] (0) -- (A');
        \draw[black,very thick] (0) -- (D');

        \draw[black,very thick] (A') arc (0:120:4);
        
        \node(label) at (1.88,0.68) {$T_1$};
        \node(label) at (1,1.73) {$T_2$};
        \node(label) at (-0.35,1.97) {$T_3$};
    \end{tikzpicture}
    \caption{This figure shows the triangular regions~$T_1,T_2,$ and~$T_3$ whose images under $\tilde{\boldsymbol{x}}_1$ form~$\Gamma_S$.}
    \label{fig:cone_param}
\end{figure}

The construction of~$\bx_S$ is then the same as in the proof of Theorem~\ref{thm:acceptable_param}, except that the initial parameterizations 
already preserve arc-length along the overlapping edges and so 
Lemma~\ref{lem:arc-length} is not necessary. The required affine transformations 
make the pull-back of the Laplace-Beltrami operator a constant coefficient 
operator on each face, and so we have found the desired parameterizations.
\end{proof}

A similar argument to that in Lemma~\ref{lem:corners} could be applied on~$W$ to
find that the expansion powers are still~$\frac{2\pi n}{\gamma_S}$. The rest of
the proof of Lemma~\ref{thm:interfaceholds}, and therefore
Theorem~\ref{thm:strongLB}, is the same. Our main result thus holds for this cone.

It is likely that our result also holds for Laplace-Beltrami problems along more
generic surfaces with cone-like singularities. We therefore present the following conjecture.

\begin{conjecture}
    Let~$\Gamma$ be a Lipschitz surface than can be written as the union of faces~$\left\{\Gamma_i\right\}_{i=1}^N$. Also suppose that the interiors of these faces are pairwise disjoint. Finally suppose that for each face~$\Gamma_i$ there is a closed triangle~$T_i$ and a bijective parameterization~$\bx_i:T_i\to \Gamma_i$. If for each~$i$, the parameterization~$\bx_i$ is smooth away from the vertices of~$T_i$ and can be written smoothly in polar coordinates around each verte, then the weak solution of the Laplace-Beltrami problem on~$\Gamma$ satisfies the interface form whenever the right hand side is in~$L^2(\Gamma)$.
\end{conjecture}

\section{Application: Surfaces of revolution with edges}
\label{sec:ode}
Having identified the correct form of the Laplace-Beltrami problem on a
piecewise smooth Lipschitz surface~$\Gamma$, we would like to develop numerical
methods to solve it. In general, this will be challenging, as care would have to
be taken to deal with the singularities in the solution near surface vertices.
In this section, we shall restrict ourselves to the case that~$\Gamma$ is a genus one
surface of revolution (e.g. the surface in Figure~\ref{fig:Harmonic_fields}), which cannot have vertices. In this case, we will able to
use a simple numerical method to solve the interface form of the
Laplace-Beltrami problem. Our method will be to use separation of variables to
formulate the Laplace-Beltrami problem as a sequence of decoupled periodic ODEs,
where the interface conditions become continuity conditions on the the solution
of the ODEs. We then solve those ODEs using an integral equation approach
which will automatically satisfy those continuity conditions.

\subsection{Separation of variables}

To begin with, denote by~$(r,\theta,z)$ the usual cylindrical coordinate
system in three dimensions. 
If~$\Gamma$ is a genus one piecewise smooth surface of revolution about
the~$z$-axis, then let~$\gamma$ denote its \emph{generating
curve} in the plane $\theta = 0$, which is assumed to be closed, piecewise smooth, and non-self intersecting.
The generating curve can be parameterized in
terms of arclength,~$\gamma : [0,L] \to \mathbb R^3$;
let its cylindrical coordinate components (for $\theta = 0$) be parameterized as
\begin{equation}
  \gamma(s) = (r(s), z(s)),
\end{equation}
where~$s \in [0,L]$ denotes arclength along the generating curve~$\gamma$.
As shown in the previous section, since the 
Laplace-Beltrami problem $\surflap u = f$ is uniquely solvable on the
space of mean-zero square-integrable functions on a piecewise smooth
surface, we can write the solution~$u$ in a Fourier expansion in
cylindrical coordinates as:
\begin{equation}
  \label{eq:fourieru}
  \begin{aligned}
    u(\bx) &= u(r,\theta,z) = u(s,\theta) \\
    &= \sum_{n=-\infty}^{\infty} u_n(s) \, e^{in\theta},
  \end{aligned}
\end{equation}
and the right hand side $f$ as:
\begin{equation}
  \label{eq:fourierf}
    f(\bx) = \sum_{n=-\infty}^{\infty} f_n(s) \, e^{in\theta}.
\end{equation}

Since a surface of revolution can only have edges at corners in the generating curve, the decomposition in~\eqref{eq:fourieru} implies that the interface
conditions become the requirement that $u_n$ and $u_n'$ are continuous at
values of $s$ that correspond the surface edges. 

Furthermore, we note that in the variables~$s$ and~$\theta$, the
Laplace-Beltrami operator takes the form:
\begin{equation}
  \label{eq:F_LB_ODEs}
\surflap =   \frac{\partial^2}{\partial s^2 } + \frac{1}{r} \frac{\partial
    r}{\partial s} \frac{\partial}{\partial s} + \frac{1}{r^2}
  \frac{\partial^2}{\partial \theta^2}.
\end{equation}
Using the decompositions in~\eqref{eq:fourieru}
and~\eqref{eq:fourierf},
and the above form of
the Laplace-Beltrami operator, we can transform the PDE into a
sequence of decoupled periodic ODEs, one for each Fourier mode~$n$. Specifically, the Fourier modes $u_n$ should be continuously differentiable, $L$-periodic, and such that
\begin{equation}
  \label{eq:odes}
  \frac{d^2 u_n}{d s^2 } + \frac{1}{r} \frac{d
    r}{d s} \frac{d u_n}{d s} - \frac{n^2}{r^2}
  u_n
  = f_n, \qquad \text{for } s \in [0,L].
\end{equation}

The solution $u$ can then be easily synthesized via its Fourier ansatz. Note
that the mean-zero condition on~$u$ for solvability of the Laplace-Beltrami
problem reduces to a condition on only~$u_0$ since every mode with~$n\neq 0$
integrates to zero:
\begin{equation}
  \int_\Gamma u = \int_0^{2\pi} \int_0^L \sum_n u_n(s) \, e^{in\theta}
  \, r(s) \, ds \, d\theta  = 2\pi \int_0^L u_0(s) 
  \, r(s) \, ds.
\end{equation}
Enforcing this mean-zero condition on~$u_0$ 
is discussed in the next section.

\subsection{A periodic ODE solver}

The separation of variables solution to the Laplace-Beltrami problem
requires solving the sequence of periodic ODEs in~\eqref{eq:odes} with
the usual periodic boundary condition: continuity in the solution and
its derivative~\cite{coddington1984,trefethen2018,eastham1973}.  In
order to solve the ODE for each Fourier mode of the Laplace-Beltrami
equation, we shall convert it into a second-kind integral
equation. Doing this will allow us to easily use adaptive
high-order quadrature methods to solve it accurately.  This conversion
is applicable to a broad class of periodic ODEs, so we shall present
the procedure in a general framework. A similar approach for the
Laplace-Beltrami problem on smooth surfaces of revolution was used
in~\cite{Epstein2019} (using a global trapezoidal discretization
scheme);  a more general adaptive approach for two-point boundary
value problems, coupled with a fast
direct solver, was detailed in the widely known work
of~\cite{lee1997}. To this end, we shall consider a method to solve
any ODE on $[0,L]$ of the form:
\begin{equation}
  \label{eq:ODE}
  u'' + pu'+qu = f, \qquad 
  u(x+L)=u(x), \qquad u'(x+L)=u'(x),
\end{equation}
where $f$, $p$, and $q$ are known periodic functions in $\Lr$ for some
$r>1$ and we are searching for a solution $u\in C^1(\Iper)$.  The boundary conditions above give that solution is $L$-periodic.
 If $q = 0$, then the solution can only be
determined up to an additive constant; in this case, an additional
constraint must be imposed to ensure well-posedness of the problem.
Usually this constraint takes the form of a linear function of~$u$,
such as
\begin{equation}
  \label{eq:constraint}
  \int_0^{L} u(x)\, w(x) \, dx=A.
\end{equation}
We address this special case where~$q= 0$ later on in this section.

In order to convert the ODE~\eqref{eq:ODE} on~$[0,L]$ into an integral
equation on the same interval,
first consider the kernel~$\GL$,
\begin{equation}
  \GL(x) = -\frac{1}{2L} \left(\text{mod}(x,L)-\frac{L}2\right)^2
  +\frac{L}{24}.
\end{equation}
It is not hard to verify that if this kernel is convolved with a
function $f$ that is mean-zero on~$[0,L]$, then the resulting function
$v=\GL* f$ solves the one-dimensional periodic Poisson equation
$v''=f$. This $v$ is in fact the unique mean-zero solution
with~$\int_0^L v = 0$, since $\GL$ is also mean-zero on this interval.
We next define the ``single layer operator''~$\SL$ via the convolution
\begin{equation}
  \SL f (x) = \int_0^L \GL(x-t) \, f(t) \, dt.
\end{equation}
Since~$\GL$ is mean-zero on~$[0,L]$, $\SL f$ is also mean-zero on the
same interval.  We now represent
the solution~$u$ to~\eqref{eq:ODE} as
\begin{equation}
  \label{eq:rep}
  u=\SL\sigma+C,
\end{equation}
for some unknown density~$\sigma$ with~$\int_0^L \sigma = 0$ and some unknown
constant~$C$ which, by construction, gives
\begin{equation}
  \int_0^L u = \int_0^L \SL \sigma + \int_0^L C  \quad \implies
  \quad C = \frac{1}{L} \int_0^L u,
\end{equation}
i.e. $C$ is the mean of the solution~$u$ on~$[0,L]$.
 We have included the constant~$C$ in the representation
in order to ensure that it is a complete representation; $\SL\sigma$
will always be a mean-zero function on~$[0,L]$, but the solution~$u$ may not be.
This representation
for $u$ also ensures that the solution is automatically periodic
since~$\GL$ is periodic.

With this representation, we have changed the problem of finding $u$
into the problem of finding~$\sigma$ and~$C$.
Inserting~\eqref{eq:rep} in~\eqref{eq:ODE} yields a Fredholm
second-kind integral equation for~$\sigma$ and $C$:
\begin{equation}
  \label{eq:new-int-eq}
  \sigma + p\SL'\sigma + q \SL\sigma+ q \, C =f,
\end{equation}
where
\begin{equation}
  \SL'\sigma(x)=\int_0^{L}\GL'(x-t) \, \sigma(t) \, dt.
\end{equation}
The above integral equation is indeed invertible when $q\neq 0$
because the underlying ODE is invertible and we simply used a
complete and unique representation for the solution~$u$. In order to
see that our transformation has led to a well-conditioned equation, we
note that~\eqref{eq:new-int-eq} has the form of
$(\mathcal I+\mathcal K)\sigma+qC=f$, where~$\mathcal I$ is the
identity operator and~$\mathcal K$ is a compact integral operator.
The integral equation is therefore of Fredholm
second-kind~\cite{Kress1983}.

If $q = 0$, then we must explicitly enforce the additional integral
condition~$\int_0^{L} u \, w =A$ from~\eqref{eq:constraint}. We shall
include this condition by simply adding it to~\eqref{eq:new-int-eq},
giving the integral equation
\begin{equation}
  \label{eq:new-int-eq2}
  \sigma + p\SL'\sigma + \int_0^{L}(\SL\sigma+C) \, w =f + A.
\end{equation}
This method of adding a linear constraint to an integral equation is
equivalent to adding a rank-one update to the original integral
operator. It is not difficult to see that this update results in an
invertible and well-conditioned equation provided the range of the
update is not contained in the range of the original integral
operator.  See~\cite{Sifuentes2015} for a discussion of this method in
the matrix equation setting. In our case, this is equivalent to asking
if there exists a mean-zero function on~$[0,L]$ that is not in the
null space of the adjoint, i.e. if there exists a mean-zero and
periodic $f$ such that $f+\SL' [pf]\neq 0$. Taking a derivative of
this expressions gives the condition that $f'+pf\neq 0$. Plugging in $f=\sin (\frac{2\pi}L x)$ or $f=\cos (\frac{2\pi}L x)$
will yield at least one example.

Next, we check that the solution of~\eqref{eq:new-int-eq} produces a~$u$
that is automatically in~$C^1(\Iper)$. To do this, we first note that
for~\eqref{eq:new-int-eq} to make sense, the solution~$\sigma$ must be
in~$L^1(\Iper)$. Applying Young's inequality then tells us that
$\SL\sigma$ and $\SL'\sigma$ are uniformly bounded, and
therefore~\eqref{eq:new-int-eq} gives that~$\sigma$ is in $\Lr$.
Finally, we note that
$\SL$ maps $\Lr\to C^1(\Iper)$. This can easily be proven using
H\"older's inequality and the fact that $\GL$ and $\GL'$ are
piecewise continuous and bounded. Knowing that $\sigma \in \Lr$ then
gives that $\SL\sigma$, and thus the solution $u$, are in
$C^1(\Iper)$. An analogous argument holds true in the~$q=0$ case as
well.

We now have a suitable second-kind integral equation form
of~\eqref{eq:ODE} and its solution will have
the expected smoothness properties. In the Laplace-Beltrami case, we have that
\begin{equation}
  p = \frac{1}{r} \frac{d r}{ds}, \qquad q = -\frac{n^2}{r^2},
\end{equation}
and when~$n = 0$ we enforce the additional constraint on~$u_0$ of
\begin{equation}
  \int_0^L u_0(s) \, r(s) \, ds = 0,
\end{equation}
which, as mentioned before, is equivalent to the mean-zero
constraint~$\int_\Gamma u = 0$~\cite{Epstein2019}. Since $r$ is a piecewise smooth function bounded away from 0, $p$ and $q$ will be in $L^r(\Iper)$ for any $r\geq 1$. 
Lastly, in light of the earlier discussion in the manuscript, the case of interest where~$f \in L^2[0,L]$ satisfies the earlier requirements.

\begin{remark}\label{rmk:Yukawa}
  The above integral equation formulation of~\eqref{eq:ODE} is not
  the only possible integral equation formulation.  The above
  derivation could easily be repeated with other kernels (i.e. Green's
  functions). One example can be found in~\cite{Epstein2019}, where they
  considered the function
  \begin{equation}
    \GY(x) = -\half e^{-|\tilde
      x|}-\frac{e^{-L}}{1-e^{-L}}\cosh(\tilde x), \qquad\text{where }
    \tilde x = \mod\left(x-\frac L2,L\right)+\frac L2,
\end{equation}
instead of the function~$\GL$ used above.
This function is the $L$-periodic Green's function for the Yukawa
problem $v''-v=f$. Since $\GY$ has the same smoothness properties as
$\GL$, the resulting integral equation will still be second-kind and
give rise to a unique solution~$u$ that is periodic and
in~$C^1(\Iper)$
whenever $p$, $q$, and $f$ are in~$L^r(\Iper)$ for some
$r>1$.

We also note that $\GY$ has a minor advantage over $\GL$ because it
has a non-zero mean. Therefore, there is no need to ensure that
$\sigma$ is a priori mean-zero, and we would not need to include the
constant~$C$ in the representation (but the mean-zero
constraint~\eqref{eq:constraint} would still need to be enforced in
the case that $q=0$). In practice, these advantages are small and we
will see later that both kernels ultimately lead to very similar
performance.
\end{remark}

\subsection{A numerical solver}
\label{sec:solver}

In this section, we shall describe a numerical solver for the integral
equations in~\eqref{eq:new-int-eq}. To summarize, our solver is
based on a Nystr\"om method for the equation using an adaptive
discretization of the interval $[0,L]$ consisting of a piecewise
16th-order Gauss-Legendre quadrature. (A 16th-order quadrature was
chosen so as to ensure a high-order discretization, of course our
method extends to any other order discretization.) An overview of
various methods for discretizing integral equations along curves in
two dimensions is given in~\cite{Hao2014}. For problems in which the
surface~$\Gamma$ is smooth (and therefore so is the generating curve
$\gamma$), we partition the interval~$[0,L]$ into uniformly sized
panels. For piecewise smooth generating curves~$\gamma$, we dyadically
refine the panels using knowledge of where the geometric singularities
occur as a function of arclength.

Let $\{P_i\}_{i=1}^M$ be a partition of the interval $[0,L]$, and
$\{x_{ij},w_{ij}\}_{j=1}^{16}$, be the 16th-order Gauss-Legendre quadrature nodes and
weights on panel~$i$.
The integral equation in~\eqref{eq:new-int-eq} is enforced at each of
the nodes $x_{ij}$:
\begin{multline}
  \label{eq:disc1}
  \sigma_{ij} + \sum_{k = 1}^M \lp p(x_{ij})  \int_{P_k} \GL'(x_{ij} -
  t) \, \sigma(t) \, dt + q(x_{ij}) \int_{P_k} \GL(x_{ij} -
  t) \, \sigma(t) \, dt \rp \\ + q(x_{ij}) \, C = f(x_{ij}),
\end{multline}
where~$\sigma_{ij}$ denotes the solution to the linear system which
will be approximately equal to~$\sigma(x_{ij})$.
We will show how to enforce the a priori mean-zero condition on
the~$\sigma_{ij}$'s below.
It remains now to replace the integrals above with discrete sums. 

First, notice that the kernel~$\GL$ is piecewise smooth with a
discontinuity in~$\GL'$ only at the origin. With this in mind, the
integrals in~\eqref{eq:disc1} corresponding to~$k \neq i$ can be
approximated using 16th-order Gauss-Legendre quadrature nodes, for
example:
\begin{equation}
  \label{eq:disc2}
  \int_{P_k} \GL(x_{ij} -   t) \, \sigma(t) \, dt \approx
  \sum_{\ell=1}^{16} w_{k\ell} \, \GL(x_{ij} -   x_{k\ell}) \,
  \sigma_{k\ell}, \qquad \text{for } k \neq i.
\end{equation}
For the ``near field'' integrals corresponding to when $k = i$
in~\eqref{eq:disc1}, standard Gauss-Legendre quadrature will fail to
yield high-order convergence due to the irregularity in the
kernel~$\GL$. With this in mind, we split the $k=i$ integrals into two
pieces precisely at the point of irregularity,~$x_{ij}$. On each of
these new panels, the integrand is smooth and standard Gauss-Legendre
quadrature can be used along with Lagrange interpolation
on~$\sigma$ to obtain implied values at the extra quadrature support
nodes.

To explain in more detail:
if $P_i = [a_i,b_i]$, then for each~$j$, we approximate this near field integral
over~$P_i$ as
\begin{equation}
  \label{eq:disc3}
  \begin{aligned}
  \int_{P_i} \GL(x_{ij} -   t) \, \sigma(t) \, dt &=
  \int_{a_i}^{x_{ij}} \GL(x_{ij} -   t) \, \sigma(t) \, dt +
  \int_{x_{ij}}^{b_i} \GL(x_{ij} -   t) \, \sigma(t) \, dt  \\
  &\approx \sum_{\ell = 1}^{16} u_{ij\ell} \,
  \GL(x_{ij} -   s_{ij\ell}) \, \tilde\sigma(s_{ij\ell})
  + \sum_{\ell = 1}^{16} v_{ij\ell} \, \GL(x_{ij} -   t_{ij\ell}) \,
  \tilde\sigma(t_{ij\ell}),
\end{aligned}
\end{equation}
where~$(s_{ij\ell},u_{ij\ell})$ is the $\ell$th Gaussian quadrature
node and weight pair on the interval~$[a_i,x_{ij}]$,
$(t_{ij\ell},v_{ij\ell})$ is the $\ell$th Gaussian quadrature node and
weight pair on the interval~$[x_{ij},b_i]$, and~$\tilde\sigma(x)$ is
the value obtained from the~$\sigma_{ij}$'s on~$P_i$ via Lagrange
interpolation to the point~$x \in P_i$. The near field integrals
in~$\SL'$ can be discretized similarly.

Finally, recall our representation for the solution~$u$: $u =
\SL\sigma + C$, where~$\int_0^L \sigma = 0$.
This assumption on~$\sigma$ must be explicitly enforced, and can
easily be done so in one of two ways:  (1) the condition can be discretized
as
\begin{equation}
  \int_0^L \sigma(s) \, ds = 0 \approx \sum_{i=1}^M \sum_{j=1}^{16} w_{ij}
  \, \sigma_{ij}
\end{equation}
and appended to the system of equations yielding a square linear
system of dimension~$16M + 1$ for the unknowns~$\sigma_{ij}$ and~$C$,
or (2) the original representation can be replaced with one of the
form
\begin{equation}
  u = \SL\left[ \sigma - \int_0^L \sigma \right] + \int_0^L \sigma.
\end{equation}
The above alternative representation ensures that the argument
to~$\SL$ has mean-zero on~$[0,L]$ and~$C$ is equated with the
integral of~$\sigma$. This approach results in a square linear system
of size~$16M$ for only the unknowns~$\sigma_{ij}$.
In our the subsequent numerical experiments, our solver implements the
latter choice of changing the representation to equate the integral
of~$\sigma$ with the integral of the solution~$u$.

Lastly, in the purely azimuthal component to the Laplace-Beltrami
problem, i.e. the $n=0$ mode, we must also enforce the mean-zero
condition on the solution~$u_0$ along the surface~$\Gamma$. A
discretization of this condition in~\eqref{eq:constraint} can be
obtained by an identical procedure as to that used in
discretizing~$\SL \sigma$; the resulting equation can be put in the
form
\begin{equation}
  \sum_{i = 1}^M \sum_{j=1}^{16} c_{ij} \, \sigma_{ij} = 0
\end{equation}
and added to each row of the~$16M \times 16M$ linear system.

\begin{remark}
The continuous Laplace-Beltrami equations are well-conditioned with
respect to the $L^2(\Iper)$ norm, $\|\cdot\|_{L^2(\Iper)}$. In order
to ensure that our discretized equations are similarly
well-conditioned in~$\ell^2$ (as an embedding of the continuous
problem), we must use a discretization method that that approximates
$\|\cdot\|_{L^2(\Iper)}$~\cite{Bremer2012}. This is especially
important when we use local adaptive refinement, as that will cause
the $\ell^2$ norm of the discretized functions to greatly diverge from
their true~${L^2(\Iper)}$ norms. We address this issue by
replacing~$\sigma_{ij}$ in the discretized linear system
with~$\sqrt{w_{ij}} \sigma_{ij}$. An equivalent linear system is
easily obtained by left and right diagonal preconditioners~(as
detailed in~\cite{Bremer2012}) with the effect that the~$\ell^2$ norm
of the discrete unknown approximates the true~$L^2$ norm of the
solution to the continuous problem:
\begin{equation}
  \left\Vert \{ \sqrt{w_{ij}} \sigma_{ij}  \} \right\Vert^2 = \sum_{ij} w_{ij}
  \sigma_{ij}^2 \approx \int_0^L \sigma^2(s) \, ds.
\end{equation}
Furthermore, the resulting linear system has a spectrum which
converges to the spectrum of the original continuous system, yielding
increased performance when using iterative solvers such as GMRES.
\end{remark}

\section{Numerical examples}
\label{sec:numerical}

In this section, we give the results of some numerical examples
demonstrating the ODE and Laplace-Beltrami solvers detailed above. In
order to compute the right hand sides~$f_n$ in
equation~\eqref{eq:odes} we discretize the original function~$f$ using
equispaced points in the azimuthal and compute~$f_n$ using the FFT (as
in~\cite{Epstein2019}). In all of the tests, the linear systems were
solved using GMRES to a relative tolerance of~$10^{-14}$. The linear
systems were of modest size, and could easily be applied in a
matrix-free fashion. All codes were written in MATLAB and no attempt
was made to accelerate the code beyond the use of the FFT.

All reported errors were estimated in the relative~$L^2$ sense,
e.g. the relative error in the solution~$u$ to an ODE was measured as
\begin{equation}
  \epsilon = \sqrt{ \frac{\sum_i w_i \, (\utrue(x_{i}) - u_i)^2 } 
           {\sum_i w_i \, \utrue^2(x_{i})  } },
\end{equation}
where we have used a single subscript~$i$ above to index
all~$16M$ values on the interval~$[0,L]$ and~$\utrue$ was either known a priori or estimated using a finer
discretization or a different numerical solver. For Laplace-Beltrami
problems on a surface~$\Gamma$, the formula above is modified to
include the proper quadrature weight for a surface integral (obtained
as a tensor product of piecewise Gauss-Legendre quadrature with the
trapezoidal rule).

Lastly, in what follows, we will use~$N_s = 16M$ to denote the total
number of discretization points on an interval~$[0,L]$ and~$N_\theta$
to denote the total number of points used in the azimuthal direction
for Fourier analysis/synthesis via the FFT.

\subsection{Comparing Greens functions}
\label{sec:Comp_fcts}

In order to compare the Green's function~$\GL$ discussed in
Section~\ref{sec:solver} with the Green's function~$\GY$ used
in~\cite{Epstein2019}, we test our ODE solver on a globally smooth
problem. We set the coefficient functions in the ODE to
\begin{equation}
  p(x)=\sin(3x)-2, \qquad  q(x)=2\sin(5x)-3
\end{equation}
and the right hand side to
\begin{equation}
  f(x)=\frac{d^2}{dx^2} e^{\sin(2x)}.
\end{equation}
These functions are smooth, periodic, and the
resulting ODE is not trivialized by a simple change of
coordinates. Since these functions are smooth, we use uniform panels
to discretize the interval.

As the exact solution for this equation is not known, we compare the
resulting solutions to the solution obtained with~\chebfun, a well-known
MATLAB package that uses global Chebyshev approximation to solve ODEs with
spectral accuracy~\cite{Platte2010}.
As a further test, we repeated this experiment with a Chebyshev-based
discretization (as opposed to the Gauss-Legendre one described
earlier).
In this test, second-kind Clenshaw-Curtis quadratures were used 
to compute the required integrals.

\begin{figure}[t]
    \centering
    \includegraphics[width=.75\linewidth]{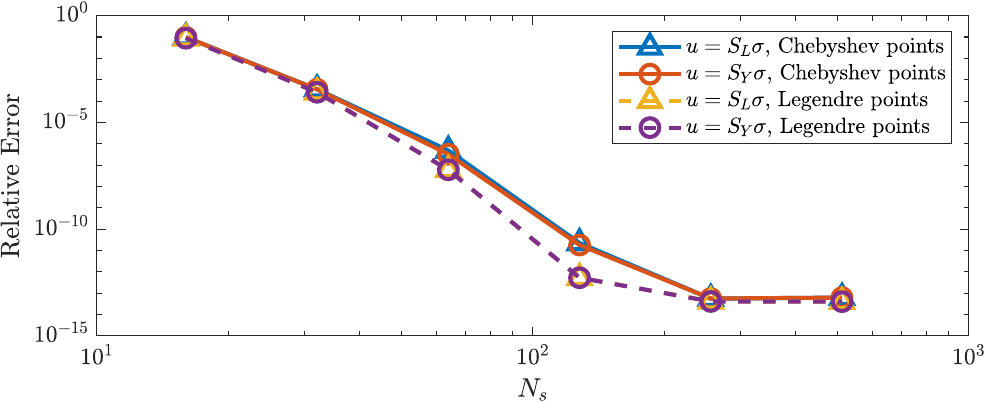}
    \caption{Convergence results for our smooth ODE solver test. The
      number of discretization points, $N_s$, is compared to the
      relative~$L^2$ difference between the computed solution and the
      solution obtained via \chebfun. The difference plateaus
      at~$10^{-12}$ as that is the accuracy reported by the \chebfun~ solver.}
    \label{fig:ODE_refinement}
\end{figure}

Figure~\ref{fig:ODE_refinement} depicts the convergence of the solution
in this
regime. We see that our method indeed converges and gives us an
accurate solution to the problem with a moderate number of
discretization points. We can also see that both choices of kernels
result in comparable accuracy, and neither was consistently better. As
such, we chose to use the Poisson kernel~$\GL$ for the remainder of
this paper. Furthermore, Figure~\ref{fig:ODE_refinement} also makes it
clear that using Gauss-Legendre points results in a more accurate
solution than is achieved with the same number of Chebyshev points,
but only slightly so and both solutions converge to near machine
precision.

\subsection{Laplace-Beltrami on a smooth surface}

In order to validate our Laplace-Beltrami solver, we constructed a surface
$\Gamma$ and right-hand side $f$ with a known solution. We did this by
constructing a smooth surface and choosing the exact solution to be a
constant plus the restriction of a smooth function $v$ defined in all
of $\bbR^3$ (following the approach in~\cite{ONeil2018}).
We then generated $f$ through the following well-known
analytic formula which applies in any neighborhood that a surface is
smooth:
\begin{equation}
  \label{eq:LB_rest}
  f=\LB \left(v|_\Gamma\right) = \Delta v-2H \frac{\partial
    v}{\partial \bs n}-\frac{\partial^2 v}{\partial \bs
    n^2} .
\end{equation}
In the above formula, $H$~is the mean curvature of~$\Gamma$ and $\bs
n$ is the normal to~$\Gamma$. A derivation of this formula may be
found in \cite{nedelec2001}. We evaluate $f$ by analytically computing
the terms in~\eqref{eq:LB_rest} based on a global parameterization of
the surface.
In this example,~$\Gamma$ is given by
a circular torus with inner radius one and outer radius two,
and we set $v$ to be the Newtonian potential centered at {$\bs x_0=(0,0.5,0.5)$}:
\begin{equation*}
    v(\bs x) = -\frac{1}{|\bs x - \bs x_0|}.
\end{equation*}
 The exact solution to this problem is given by
 $u=v|_\Gamma-\frac{1}{|\Gamma|}\int_\Gamma v$, where~$|\Gamma|$ is
 the surface area of~$\Gamma$. This solution and the corresponding
 right hand side are shown in~Figure~\ref{fig:surface_plots}. We
 conducted a convergence test by varying~$N_\theta$ and~$N_s$. We can
 see from the relative errors in~Figure~\ref{fig:Smooth_LB_test} that our
 method is capable of accurately solving this problem, and therefore
 the solver is working as expected.

\afterpage{%
  \vspace{5\baselineskip}
  \begin{figure}[h]
   \centering
   \begin{subfigure}{0.45\linewidth}
     \centering{\includegraphics[width=.95\linewidth]{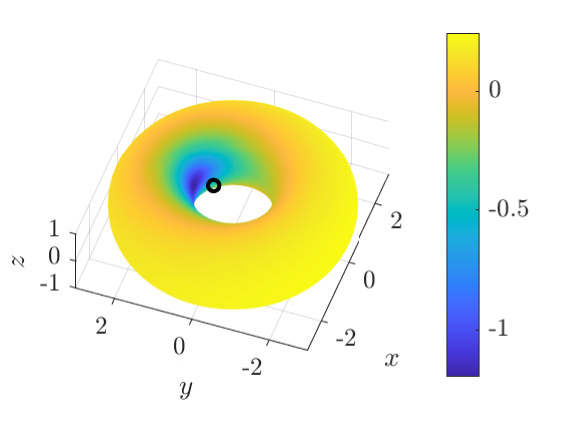}}
     \caption{The solution of $\LB u=f$. }
   \end{subfigure}
   \quad
   \begin{subfigure}{0.45\linewidth}
     \centering{\includegraphics[width=.95\linewidth]{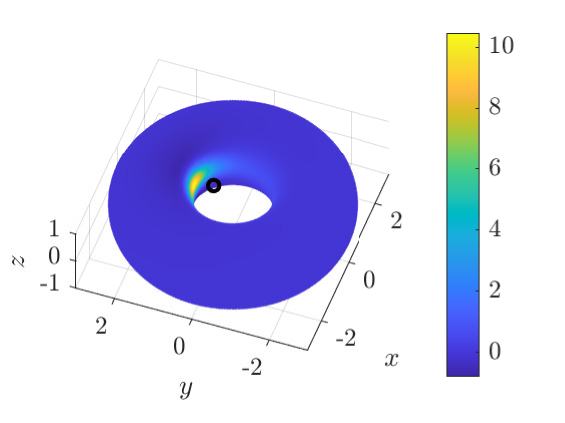}}
     \caption{The right hand side of $\LB u=f$.}
   \end{subfigure}
   \caption{The solution and right hand side of $\LB u=f$, where $u$
     was chosen to be the mean-zero restriction of the Newtonian
     potential to~$\Gamma$. The Newtonian potential is centered at the
     black circledot.}
   \label{fig:surface_plots}
 \end{figure}
  \vspace{5\baselineskip}
\begin{figure}[h]
    \centering
    \includegraphics[width=0.75\linewidth]{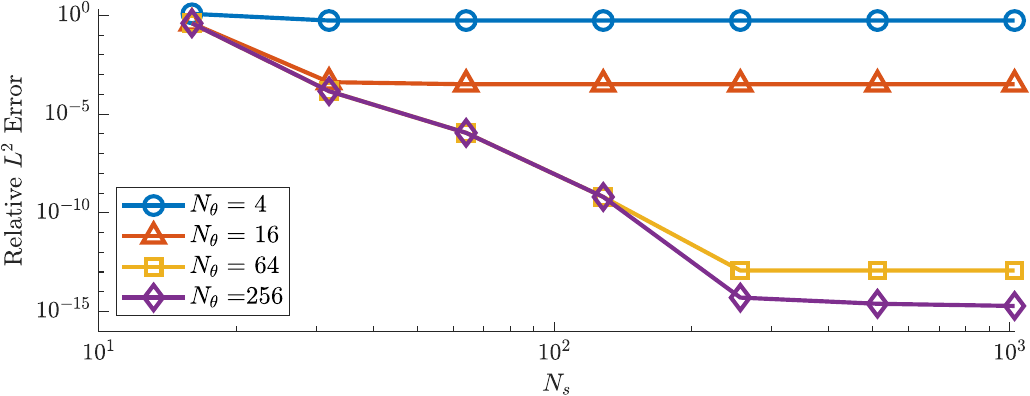}
    \caption{The errors of the computed solution in our circular torus
      tests. $N_s$ is the number of points around the generating curve
      and $N_\theta$ is the number of points used in the
      $\theta$-direction. \label{fig:Smooth_LB_test}}
\end{figure}
\clearpage
}

\subsection{Singular surface test}
\label{sec:singular}

In this experiment, we tested our Laplace-Beltrami solver on a
non-smooth surface with right-hand sides that are irregular at
surface edges. We set the surface~$\Gamma$ to be a square toroid: the
surface resulting from revolving a unit square about the
$z$-axis. Some care must be taken when discretizing this surface
because of the geometric singularities. It is necessary to ensure that panel
boundaries in our discretization of the generating curve coincided
with the surface edges, as the coefficient functions in
\eqref{eq:F_LB_ODEs} are singular at the edges of $\Gamma$. The
coarsest possible discretization thus had one panel per face. In order
to resolve the singularity in the right-hand side, we dyadically
refined this coarse discretization into the edge where $f$ is
singular. We denote the width of the finest panel as $h_{\text{final}}$ and
study how the errors depend on it. In the $\theta$-direction, our
surface and right hand sides were smooth, so we simply set
$N_\theta=10$.

In this example, we are not able to generate an exact solution and
right-hand side through the same method as we used in the smooth
surface case, as the restriction of a smooth function to our surface
would not have been in~$C^1(\Gamma)$. Instead, we generate them using \chebfun. Specifically, we specify that the true
solution satisfies
\begin{equation}
  \label{eq:unos}
  \partial_ s^2u(\theta,s)=\Theta(\theta)S(s),
\end{equation}
where~$\Theta$ and~$S$ are specified functions. This allows us to use \chebfun's anti-differentiation routine to compute~$u,\partial_s u$, and $f =\LB u$ in $\theta$-Fourier
space. We note that this is
different than our methodology in Section~\ref{sec:Comp_fcts}, where
we instead specified the form of $f$ and used \chebfun~to solve the ODE
for~$u$. We choose to use anti-differentiation here as it becomes
trivial to ensure that $u$ is mean-zero on~$\Gamma$. Furthermore, this
scheme does not require us to
solve a singular ODE in order to compute the reference solution.

Explicitly, we find each Fourier mode~$n$ of~$u$ and~$f$ by performing the
following calculations:
\begin{equation}
  \partial_s u_n = \int \partial_s^2 u_n
  -\frac{1}{4}\int_0^{4} \lp \int \partial_s^2 u_n \rp,
\end{equation}
and then for~$n\neq 0$
\begin{equation}
u_n=   \int \partial_s u_n -\frac{1}{4}\int_0^{4}\left(\int \partial_s
u_n\right),
\end{equation}
and for $n=0$
\begin{equation}
  u_0 = \int \partial_s u_0 - \frac{1}{4}\int_0^4 w\left(\int \partial_s u_0\right).
\end{equation}
The right hand side is then computed (for all $n$) as
\begin{equation}
  f_n = \partial_s^2 u_n + \frac{\partial_s r}{r} \partial_su_n +
  \frac{n^2}{r^2} u_n.
\end{equation}
Above, we use $\int$ to denote anti-derivative.

We tested the solver with several choices of~$S$
in~\eqref{eq:unos}. As a smooth test we set $S(s) = \cos(\pi s/2)$. We
also tested several singular cases of the form $S(s)=|s-2|^\alpha$,
where we set $\alpha=-\frac13,-\frac12,$ and $-\frac34$. Here~$s=2$ is the
arclength parameter corresponding to the top inner edge, so $u$~is
non-smooth at that edge (and the one directly across from it). For all
of our tests, we set~$\Theta(\theta)=\sin(3\theta)$. 

Cross sections of the resulting solutions are displayed in Figure~\ref{fig:mult_alps_soln}. As expected, the solutions are continuously differentiable, with local extrema at~$s=2$. Further, as~$\alpha$ approaches -1 and the minimum~$r$ such that~$f_n\in L^r(\Iper)$ shrinks, the point of the extrema becomes sharper and the solution approaches the boundary of~$C^1(\Iper)$. This experiment is thus a good stress-test of our solver and will demonstrate how the solver behaves on problems that barely satisfy its basic requirements.

Figure~\ref{fig:non_smooth_surface_err} shows how the error in each test
problem depended on the extent of the dyadic refinement. We see that
the smooth case was resolved with a single panel on each face. From
this, we can see that discontinuous coefficient functions did not
prevent our method from computing accurate solutions with relatively
few unknowns. In the singular tests, we can see that sufficient
refinement was all that was necessary to achieve an accurate
solution. Thus, the limited resolution of the singularity in the
finest panel was the source of the dominant error in the
solution. Furthermore,~Figure~\ref{fig:non_smooth_surface_alpha}
demonstrates that the error in fact decayed
as~$\mathcal O(h_{\text{final}}^{1+\alpha})$, which was the order
of accuracy in our
evaluation of the left hand side of the integral equation. This fact
is further discussed in Appendix~\ref{sec:error_bound}.

\begin{figure}[t]
    \centering
    \includegraphics[width=.7\linewidth]{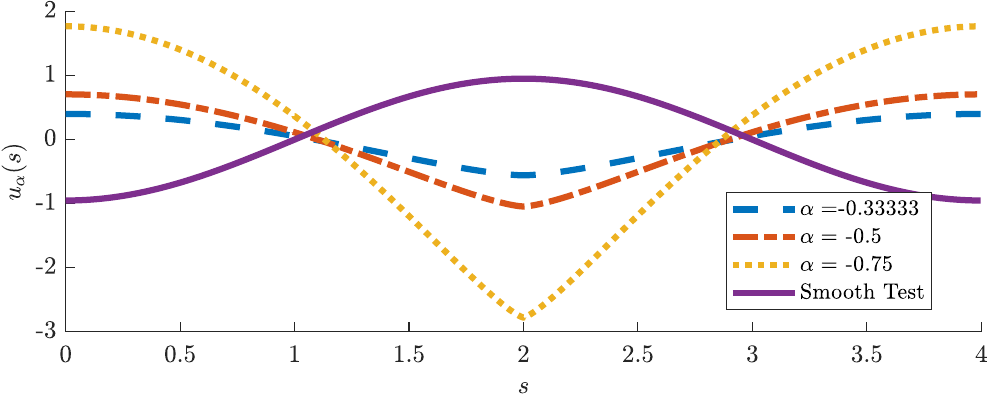}
    \caption{This figure shows the four solutions that were used to
      test our solver in the singular geometry of a square
      toroid. Each solution is plotted around the cross section
      $\theta = \frac{2\pi}5$.}
    \label{fig:mult_alps_soln}
\end{figure}

\afterpage{%
\begin{figure}[h]
  \centering
  \begin{subfigure}{0.45\linewidth}
    \centering{\includegraphics[width=.95\linewidth]{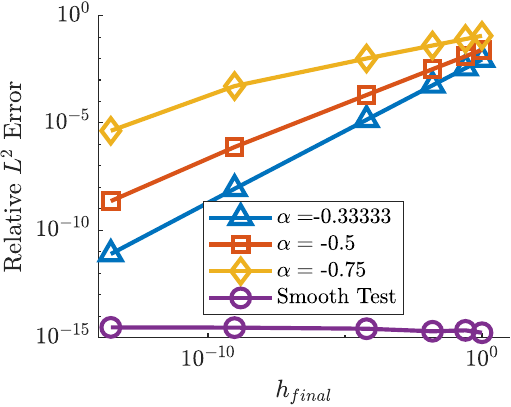}}
    \caption{The relative~$L^2$ errors of the square toroid tests.
      \label{fig:non_smooth_surface_err}}
  \end{subfigure}
  \quad 
  \begin{subfigure}{0.45\linewidth}
    \centering{\includegraphics[width=.95\linewidth]{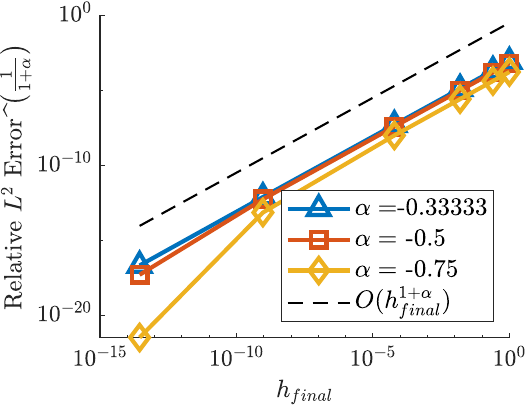}}
    \caption{The relative $L^2$ errors of the square toroid tests
      raised to the power~$\frac{1}{1+\alpha}$. The dashed line
      indicates the scaling
      $\mathcal O(h_{\text{final}}^{1+\alpha})$.
      \label{fig:non_smooth_surface_alpha}}
  \end{subfigure}
  \caption{The figures show how the relative error of the square
    toroid tests depended on the extent of the refinement. Each marker
    indicates a trial.}
\end{figure}
\vspace{5\baselineskip}
\begin{figure}[h]
  \centering
  \begin{subfigure}{0.45\linewidth}
    \centering{\includegraphics[width=.95\linewidth]{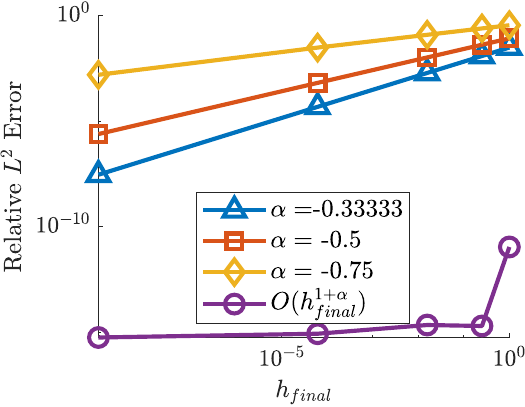}}
    \caption{The relative $L^2$ differences in the self-convergence
      study on the square
      toroid. \label{fig:non_smooth_surface_err_self}}
  \end{subfigure}
  \quad 
  \begin{subfigure}{0.45\linewidth}
    \centering{\includegraphics[width=.95\linewidth]{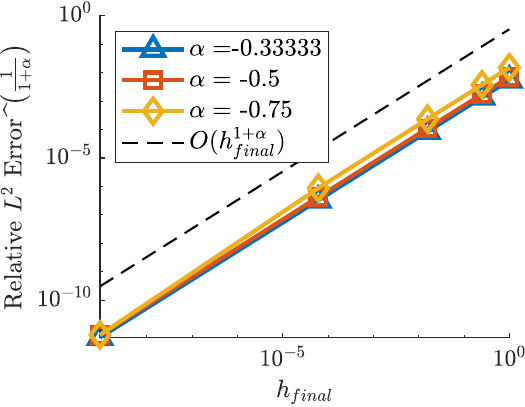}}
    \caption{The relative $L^2$ difference raised to the power
      $\frac{1}{1+\alpha}$. The dashed line indicates the scaling
      $\mathcal O(h_{\text{final}}^{1+\alpha})$.\label{fig:non_smooth_surface_self_alpha}}
  \end{subfigure}
  \caption{These figures show how the relative difference between the
    most refined solution and the other solutions depends on their
    refinement in the self-convergence study.}
\end{figure}
\clearpage
}

As a further verification, we performed a self convergence study that
did not make use of \chebfun. To do this, we directly set
$f(\theta,s)=\Theta(\theta)S(s)$, where $\Theta$ and $S$ were chosen
to be the same as in the previous
experiment. Figure~\ref{fig:non_smooth_surface_err_self} shows the
relative difference between the solution on the finest grid and the
solutions on other grids, for each of the choices of $S$. Most of the
observations from the previous experiment apply here, except that in
the case where $S(s)=\cos(\pi s/2)$. In this experiment, the corresponding solution was
not smooth, so two panels on each face were necessary in order to
accurately resolve it.

 \subsection{Harmonic vector field computation}

 The surfaces of revolution that we have considered so far are
 non-trivial and of genus one; it is well-known that they support
 a two-dimensional linear space of harmonic vector fields, i.e.
 square-integrable
 tangential vector fields~$\bs H\in \bs L^2_t(\Gamma)$ such that
 $\divg \bs H= 0$ and $\divg (\bs n \times \bs H)= 0$ in $\Hm$. For
 all surfaces of revolution, an orthogonal basis for these harmonic
 vector fields is analytically given by
 \begin{equation}
   \label{eq:basis}
   \bs H_1 = \frac{1}{r} \shat, \qquad \text{and} \qquad
   \bs H_2=\frac{-1}r \thetahat,
 \end{equation}
 where~$\shat$ and~$\thetahat$ denote unit vectors along the generating curve
 and in the azimuthal direction, respectively~\cite{Epstein2013,Epstein2019}.
 This fact may be easily verified by direct calculation. As a further test of our
 Laplace-Beltrami solver, we shall use it to construct a basis for the space of harmonic
 vector fields along the same surface~$\Gamma$ as in the previous experiment, a square toroid. The basis fields~$\bs H_1$ and~$\bs H_2$ for this surface are shown in~Figure~\ref{fig:Harmonic_fields}.

In order to compute a basis for the harmonic vector fields, we shall make use
of the Hodge decomposition of a general vector field along~$\Gamma$, which was
introduced in section~\ref{sec:intro}. This decomposition splits a tangential
vector field $\bs F$ into a curl-free component $\bs F_{\text{cf}}$ (i.e. one
where $\divg (\bs n \times \bs F_{\text{cf}} )=0$), a divergence-free component
$\bs F_{\text{df}}$, and a harmonic component $\bs H$. The Hodge decomposition
of a vector field $\bs F\in \bs{H}^1_t(\Gamma)$ can be written explicitly as
\begin{equation}\label{eq:Hodge}
  \bs F = \surfgrad \alpha + \bs n \times \surfgrad \beta +
  \bs H ,
\end{equation}
where~$\alpha$ and~$\beta$ are mean-zero scalar functions defined
on~$\Gamma$. See~\cite{Imbert-Gerard2017}, for example, for a more detailed
discussion of this representation in similar genus one geometries.
Taking the surface divergence of~\eqref{eq:Hodge}, as well as the
surface curl~\eqref{eq:Hodge}, shows
that~$\alpha$ and~$\beta$ must satisfy
\begin{equation}
  \LB  \alpha = \divg \bs F, \qquad \text{and} \qquad
  \LB \beta  = -\divg(\bs n\times \bs F).
 \end{equation}
 We note that because $\bs F\in\bs{H}^1_t$, the right hand sides $\divg \bs F$ and $\divg(\bs n\times\bs F)$ will be in $L^2(\Gamma)$ and mean-zero. The Laplace-Beltrami problems are thus well posed and there exist unique mean-zero $\alpha$ and $\beta$ satisfying the interface form with the smoothness described in Section~\ref{sec:interface_form_hols}.

With the above Hodge decomposition in mind, it becomes clear how to
compute examples of Harmonic vector fields~$\bs H$: we can simply
choose a tangential vector field~$\bs F$ and subtract off the
components $\surfgrad \alpha$ and $\bs n \times \surfgrad \beta$. In
order to make this numerically feasible, we shall restrict our choice
of $\bs F$ to be smooth on each face of $\Gamma$ and continuous across each edge of $\Gamma$. We may then compute
$\divg \bs F$ and $\divg(\bs n\times \bs F)$ pseudo-spectrally by
first Fourier decomposing~$\bs F$ as:
\begin{equation}
  \begin{aligned}
    \bs F &= F^s \, \shat + F^\theta \, \thetahat \\
    &= \sum_n \lp F^s_n \, \shat + F^\theta_n \, \thetahat \rp,
  \end{aligned}
\end{equation}
and then by applying the following formula mode-by-mode:
\begin{equation}
  \label{eq:Four-div}
  \divg \bs F = \frac{ d }{d s} F^s +\frac{1}{r}\frac{d r}{d s}F_s
  + \frac{1}{r} \frac{d}{d\theta} F^\theta.
\end{equation}
In order to compute~$dF_n^s/ds$, we interpolate~$F^s_n$ onto Chebyshev
panels in arclength along the generating curve and use Chebyshev
differentiation. Having computed the divergences (i.e. the right hand
side to a Laplace-Beltrami problem), we use our method to solve the
Laplace-Beltrami equation for~$\alpha$ and~$\beta$. Next, we compute
the surface gradient of~$\alpha$ and~$\beta$ (again, mode-by-mode) through
the formula
\begin{equation}
  \surfgrad u =  \frac{ d u}{d s} \shat + \frac{1}{r} \frac{d
    u}{d\theta}\thetahat.
\end{equation}
Note that differentiation with respect to~$\theta$ is merely
multiplication by~$(in)$ in Fourier-space. Lastly, to
compute~$d \alpha/d s$ and~$d\beta/ds$ we note that we already know~$\alpha$ and $\beta$ as the
solution of a system of integral equations with the representation,
for example,~$\alpha_n = \SL \sigma_n+C$. We can therefore use the
formula~${d \alpha_n}/{ds}=\SL'\sigma_n$ to easily compute this
quantity via quadrature on the integral
representation. Once all terms are computed for each mode, we can
synthesize the Fourier series and evaluate the harmonic component as
\begin{equation}
  \bs H = \bs F - \surfgrad \alpha - \bs n \times \surfgrad
  \beta.
\end{equation}
As a measure of accuracy, we then project~$\bs H$ onto the
basis~$\bs H_1$ and $\bs H_2$ in~\eqref{eq:basis} and look at the
$L^2(\Gamma)$ norm of the remainder relative to the norm of $\bs F$ to
determine if the computed field lies in the space of harmonic vector
fields.

As a test field, we computed the harmonic component of $\bs F = r
\shat + r^{-2} \thetahat$. With two panels per face, the relative
$L^2$ norm of the remainder was less than $10^{-14}$. We also
validated our code by computing the harmonic components of the exact
basis~$\{\bs H_1,\bs H_2\}$. We found that the basis is within machine
precision of being harmonic.

\begin{figure}[t]
  \centering
\begin{subfigure}{0.45\linewidth}\centering{\includegraphics[width=\linewidth]{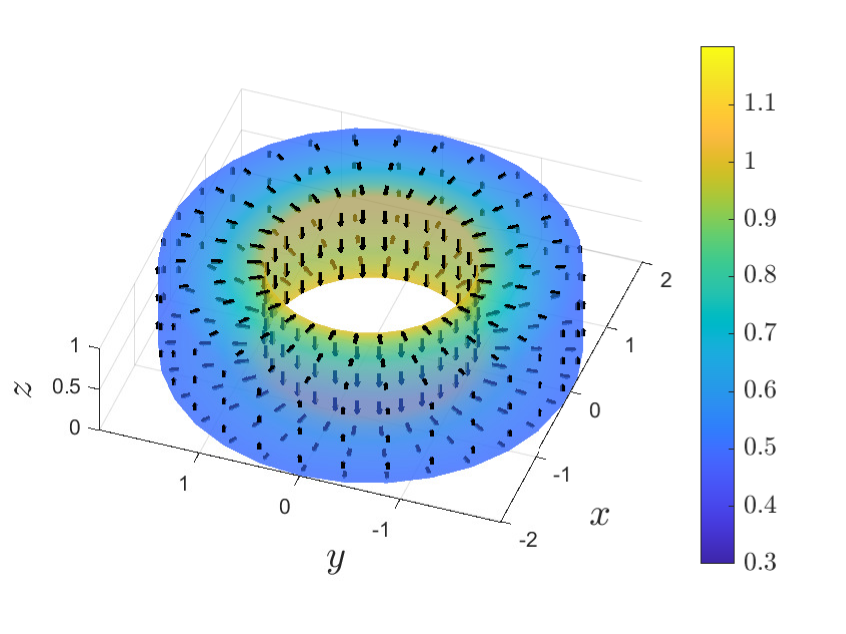}}
\end{subfigure}\quad 
\begin{subfigure}{0.45\linewidth}\centering{\includegraphics[width=\linewidth]{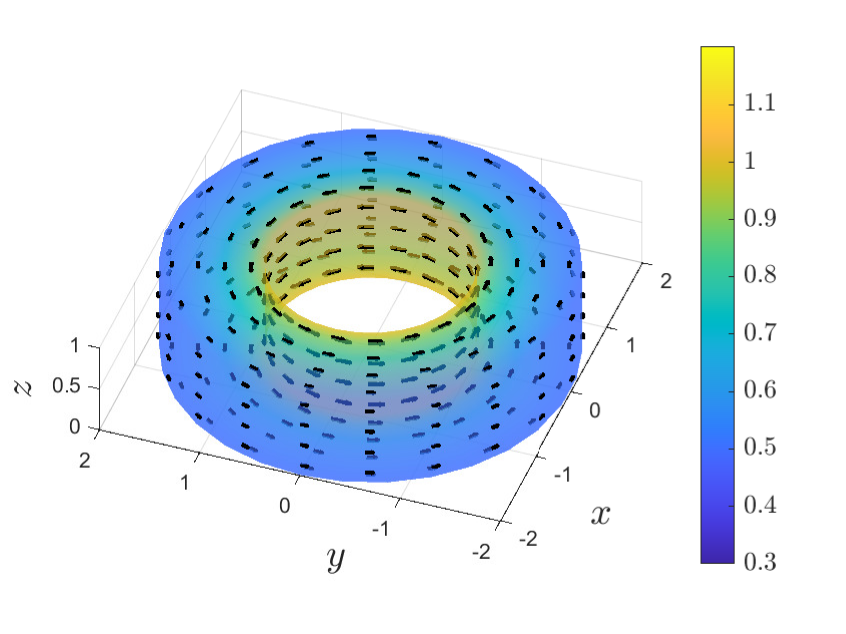}}
\end{subfigure}
\caption{The two basis harmonic vector fields for a square toroid $\bs H_1 = \frac{1}{r} \shat$ (left) and $\bs H_2=\frac{-1}r \thetahat$ (right). Color is used to indicates the magnitude of the field. }\label{fig:Harmonic_fields}\end{figure}

\section{Conclusions}
\label{sec:conclusions}

In this work, we reformulated the Laplace-Beltrami problem on
piecewise smooth surfaces as a collection of smooth problems on each
face combined with interface conditions at surface edges. To
summarize, if the right hand side of a Laplace-Beltrami problem is in
$L^2(\Gamma)$, then the solution in the usual weak sense will be well-behaved at the surface edges. Furthermore, we analytically computed an expansion of
such solutions in corners having conic angle of~$\gamma$; the leading
order term is of the form~$r^{\frac{2\pi}\gamma}$, where~$r$ is the
distance to the corner along the surface.

Furthermore, we used this reformulation to develop a numerical method
that solves the Laplace-Beltrami problem on piecewise smooth surfaces
of revolution. The numerical results support the theoretical results
of the paper. This method converted each Fourier mode of the
Laplace-Beltrami problem into a second-kind integral equation that automatically satisfied the interface conditions and could be accurately solved using standard numerical techniques for
integral equations. The integral equation formulation for solving the
associated one-dimensional periodic ODEs can be easily generalized to
any second-order periodic ODE with coefficients and right hand side
in~$L^{r}(\Iper)$ for some~$r>1$, even if they are non-smooth.

This Laplace-Beltrami solver, and the experiments used to verify it,
demonstrated the ability to easily obtain high-order accuracy for the
problem on piecewise smooth surfaces. However, this specific solver is
limited to surfaces of revolution that are separated from their axis
of rotation. In future work, we plan to develop a new integral
equation based solver that can be applied to a more general classes of
surfaces, for example piecewise smooth surfaces specified by a
collection of charts with no symmetry assumptions at all. This will be based on a parametrix
method~\cite{hormander3}, similar to the approach
in~\cite{kropinski2014fast}. Work in these directions is ongoing.

\section*{Acknowledgments}

The authors would like to thank Charlie Epstein for several useful
discussions.

\appendix

\section{Error bounds for singular data}
\label{sec:error_bound}

Here we discuss the error in our method for computing the left hand
sides of~\eqref{eq:new-int-eq} in the singular surface test. In this
test, we chose the solution to satisfy~$u''(s)=\Theta(\theta)
|s-2|^\alpha$. We will see that the error
is~$\mathcal O(h_{\text{final}}^{1+\alpha})$, where $h_{\text{final}}$ denotes
the width of the most refined panel. This will imply that the error
in~$\sigma$ has the same order.

In order to evaluate the left hand side of~\eqref{eq:new-int-eq}, we
must evaluate integrals of the form $\int_0^4 K\, \sigma_n ds$ for
various kernels~$K$. This integral is challenging to compute because
$\sigma_n=u_n''$, and therefore~$\sigma_n(s) = C|s-2|^\alpha$. We shall
assume for this discussion that $K$ is smooth and bounded on each
face. In reality, the kernel will only be piecewise smooth and we
will the use panel splitting idea in Section~\ref{sec:solver}
to accurately compute the integral. However, for the sake of clarity,
we omit these details in this discussion.

To study the error, we consider the integral over an example face:
$s\in (2,4)$. We split the integral into two pieces, one over the
finest panel $(2,2+h_{\text{final}})$ where $\sigma_n$ is singular,
and one over the rest of the panels where~$\sigma_n$ is smooth. On the
finest panel, $K$~may be approximated as having its value at $s=2^+$
since $h_{\text{final}}$ is sufficiently small. The integral thus
becomes
\begin{equation}
    \int_{2}^{2+h_{\text{final}}} K \, \sigma_n \, ds \approx
    h_{\text{final}} \int_0^{1} K(2^+) \, C \, |h_{\text{final}}
    \tilde s|^\alpha \,
    d\tilde s=C K(2^+) \frac{h_{\text{final}}^{\alpha+1}}{1+\alpha}.
\end{equation}
 If we apply our quadrature rule to the integral we obtain
 \begin{equation}
   h_{\text{final}}  \sum_{i=1}^{16} w_i C K(2^+)(h_{\text{final}}
   x_i)^\alpha  =  CK(2^+) h_{\text{final}}^{\alpha+1} \sum_{i=1}^{16}
   w_i x_i^\alpha,
\end{equation}
where the $x_i$'s and $w_i$'s are the standard 16th-order
Gauss-Legendre quadrature points and weights. The error on this panel
thus approaches
\begin{equation}
  |CK(2^+)| h_{\text{final}}^{\alpha+1}  \left|\frac{1}{1+\alpha} -
  \sum_{i=1}^{16} w_i x_i^\alpha   \right|.
\end{equation}
Since the function $x^{\alpha}$ is not integrated exactly by
Gauss-Legendre quadrature, the error on the finest panel will be
$\mathcal O(h_{\text{final}}^{1+\alpha})$.

On the panels where the function is smooth, we use the standard formula
for the error resulting from applying $k$th-order
Gauss-Legendre quadrature to integrate a function~$f$ on the
interval~$(a,b)$, see \S 5.2 of~\cite{Kahaner1989}:
\begin{equation}
  \left\vert \int_a^b f(x) \, dx - \sum_{j=1}^k w_j \, f(x_j)
  \right\vert = 
    \frac{(b-a)^{2k+1}(k!)^4}{(2 k+1)(2k)!^3}f^{(2k)}(\xi), \quad
    \text{for some } \xi\in(a,b).
\end{equation}
If we let $h_0,\ldots,h_N=h_{\text{final}}$ be the panel widths in our
dyadic refinement, then by definition $h_i= h_0 2^{-i}$. Since $K$ is
smooth on the interval $(2,4)$, the dominant term in
$(K\sigma_n)^{(2n)}$ will be $CK(s)|s-2|^{\alpha-2k}$. On the panel of
width $h_i$, this term may be bounded by $|C|\max |K|
h_i^{\alpha-2k}$, since that panel is a distance $h_i$ away from the
singularity. The error on that panel is thus bounded by
\begin{equation}
   |C|\max |K| \frac{h_i^{2k+1}(k!)^4}{(2 k+1)(2k)!^3} h_i^{\alpha-2k}.
\end{equation}
 Summing our error over all of the smooth panels gives the bound
\begin{equation}
    |C|\max |K|\frac{h_0^{1+\alpha}(k!)^4}{(2 k+1)(2k)!^3}\sum_{i=0}^N
    2^{(1+\alpha)(-i)}\leq |C|\max |K|
    \frac{h_{0}^{1+\alpha}(k!)^4}{(2 k+1)(2k)!^3}
    \frac{1}{1-2^{1+\alpha}}.
\end{equation}
Since we do not apply our scheme to the case where $\alpha$ is
exponentially close to -1, and we are using $k=16$, this error will be
well below machine precision. The error from the final panel thus
dominates the error in the smooth panels, and therefore the error in
computing the left hand side of~\eqref{eq:new-int-eq} will be
$\mathcal O(h_{\text{final}}^{1+\alpha})$.

\newpage
\bibliography{main.bib}

\end{document}